\documentclass[final,a4paper,11pt]{article}
\usepackage{graphicx}
\usepackage{mathptmx}
\usepackage{amssymb,amsmath, amsfonts,amsthm}
\usepackage{a4wide}
\usepackage{color}
\usepackage{tabularx}
\usepackage{url}
\usepackage{tikz}
\usepackage{pgfplots}
\pgfplotsset{compat=1.14}
\usepackage{comment}
\usepackage{cases}

\newcommand{\email}[1]{{\tt #1}}
\newcommand{\R}{\mathbb{R}}

\newcommand{\norm}[1]{\|#1\|}

\newcommand{\dist}[1]{{\rm dist}(#1)}

\newcommand{\mv}{\,\mid\,}
\newcommand{\B}{{\cal B}}
\newcommand{\C}{{\cal C}}
\newcommand{\I}{{\cal I}}

\newcommand{\G}{{\cal G}}

\newcommand{\T}{{\cal T}}

\newcommand{\F}{{\cal F}}

\newcommand{\skalp}[1]{\langle #1\rangle}

\newcommand{\xb}{\bar x}
\newcommand{\yb}{\bar y}

\newcommand{\oo}{o}

\newcommand{\argmin}{\mathop{\rm arg\,min}}

\newcommand{\inn}{{\rm int\,}}

\newcommand{\gph}{\mathrm{gph}\,}
\newcommand{\dom}{\mathrm{dom}\,}
\newcommand{\tto}{\rightrightarrows}
\newcommand{\Limsup}{\mathop{{\rm Lim}\,{\rm sup}}}
\newcommand{\myvec}[1]{\left(\begin{array}{c}#1\end{array}\right)}

\DeclareSymbolFont{largesymbolsA}{U}{txexa}{m}{n}
\DeclareMathSymbol{\varprod}{\mathop}{largesymbolsA}{16}

\newcommand{\ssstar}{semismooth$^{*}$ }

\newcommand{\rge}[1]{{\rm rge\;}(I-#1,#1)}

\newlength{\myAlgBox}
\newtheorem{theorem}{Theorem}[section]
\newtheorem{proposition}[theorem]{Proposition}
\newtheorem{remark}[theorem]{Remark}
\newtheorem{lemma}[theorem]{Lemma}
\newtheorem{corollary}[theorem]{Corollary}
\newtheorem{definition}[theorem]{Definition}

\newtheorem{algorithm}{Algorithm}

\title{On the application of the semismooth* Newton method to \\
variational inequalities of the second kind}
\author{Helmut Gfrerer\thanks{Institute of Computational Mathematics, Johannes Kepler University
Linz, A-4040 Linz, Austria; \email{helmut.gfrerer@jku.at}}
 \and   Ji\v{r}\'{i} V. Outrata\thanks{Institute of Information Theory and Automation, Czech Academy of
 Sciences, 18208 Prague, Czech Republic, and Centre for
              Informatics and Applied Optimization, Federation University of Australia, POB 663,
              Ballarat,  Vic 3350, Australia,  \email{outrata@utia.cas.cz}}
\and{Jan Valdman\thanks{Institute of Information Theory and Automation, Czech Academy of
 Sciences, 18208 Prague, Czech Republic, and Institute of Mathematics, Faculty of Science,
University of South Bohemia,
37005~\v{C}esk\'{e}~Bud\v{e}jovice,
Czech Republic, \email{jan.valdman@utia.cas.cz}}}}
\date{}

\begin{document}
\maketitle
{\footnotesize
{\bf Abstract.}
The paper starts with a concise description of the recently developed semismooth* Newton method for the solution of general inclusions. This method is then applied to a class of  variational inequalities of the second kind. As a result, one obtains an implementable algorithm exhibiting a local superlinear convergence. Thereafter we suggest several globally convergent hybrid algorithms in which one combines the semismooth* Newton method with selected splitting algorithms for the solution of monotone variational inequalities. Their efficiency is documented by extensive numerical experiments.


}

\section{Introduction}
In \cite{GfrOut19a} the authors have proposed the so-called semismooth* Newton method for the numerical solution of a general inclusion
\begin{equation}\label{eq-100}
0 \in H(x),
\end{equation}
where $H:\mathbb{R}^{n} \rightrightarrows   \mathbb{R^{n}}$ is a closed-graph multifunction. The aim of this paper is to work out this Newton method for the numerical solution of the {\em generalized equation} (GE)
\begin{equation}\label{EqVI2ndK}
0 \in H(x):= f(x)+\partial q (x),
\end{equation}
where $f:\mathbb{R}^{n} \rightarrow \mathbb{R}^{n}$ is continuously differentiable, $q:\mathbb{R}^{n} \rightarrow \overline{\mathbb{R}}$ is proper convex and lower-semicontinuous (lsc) and $\partial$ stands for the classical Moreau-Rockafellar subdifferential. It is easy to see that GE (\ref{EqVI2ndK}) is equivalent with the variational inequality  (VI):

Find $\bar{x} \in \mathbb{R}^{n}$ such that
\begin{equation}\label{eq-102}
\langle f(\xb), x-\bar{x}\rangle + q(x)-q(\bar{x})\geq 0 ~ \mbox{ for all } ~ x \in \mathbb{R}^{n}.
\end{equation}

The model (\ref{eq-102}) has been introduced in \cite{Gl84} and coined the name {\em VI of the second kind}. It is widely used in the literature dealing with equilibrium models in continuum mechanics cf., e.g., \cite{Hasl} and the references therein. For the numerical solution of GE (\ref{EqVI2ndK}), a number of methods can be used ranging from nonsmooth optimization methods (applicable when $\nabla f$ is symmetric) up to a broad family of splitting methods (usable when $H$ is monotone), cf. \cite[Chapter 12]{FaPa03}. If GE (\ref{EqVI2ndK}) amounts to stationarity condition for a Nash game, then also a simple coordinate-wise optimization technique can be used, cf. \cite{KaSch18} and \cite{OV}. Concerning the Newton type methods, let us mention, for instance, the possibility to write down GE (\ref{EqVI2ndK}) as an equation on a monotone graph, which enables us to apply the Newton procedure from \cite{Rob11}. Note, however, that the subproblems to be solved in this approach are typically rather difficult. In other papers the authors reformulate the problem as a (standard) nonsmooth equation which is then solved by the classical semismooth Newton method, see, e.g., \cite{ItKu09, XiLiWeZh18}.

 As mentioned above, in this paper we will investigate the numerical solution of GE (\ref{EqVI2ndK}) via the semismooth* Newton method developed in \cite{GfrOut19a}. This method is based on an important property, which is called semismoothness* and is closely related to the semismoothness property introduced in \cite{Mif77} and \cite{QiSun93}. In contrast to the Newton methods by Josephy the multi-valued part of (\ref{EqVI2ndK}) is also approximated and, differently to some other Newton-type methods, this approximation is provided by means of the graph of the limiting coderivative of $\partial q$. To facilitate the computations, we identify a certain linear structure inside the coderivative of the subdifferential mapping $\partial q$. In this way the computation of the Newton direction
 reduces to the solution of a linear system of equations. To ensure local superlinear convergence one needs merely metric regularity of the considered GE around the solution.

 The plan of the paper is as follows. After the preliminary Section 2 in which we collect the needed notions from modern variational analysis, the semismooth* Newton is described and its convergence is analyzed (Section 3). Thereafter, in Section 4 we develop an implementable version of the method for the solution of GE \eqref{EqVI2ndK} and show its local superlinear convergence under mild assumptions. Section 5 deals with the issue of global convergence. First we suggest a heuristic modification of the method from the preceding section which exhibits very good convergence properties in the numerical experiments. Thereafter we show global convergence for a family of hybrid algorithms where, under monotonicity assumptions, one combines the semismooth* Newton method with various frequently used splitting methods. The resulting algorithms show better convergence properties than the underlying splitting methods themselves. In fact, using this semismooth* hybrid approach we can solve problems, where the pure splitting methods failed. One possible explanation of this phenomenon consists in the fact that for the convergence of the semismooth* Newton method one needs merely the metric regularity and not monotonicity. Finally, the concluding Section 6 is devoted to the presentation of numerical experiments. It contains a low-dimensional Nash equilibrium which admits both a monotone as well as a non-monotone variant. To its computation we apply the implementation developed in Section 4. Thereafter we report about a rather extensive testing of the heuristic and the hybrid methods by means of a specially constructed family of medium-scale GEs.

 The following notations is employed, $\B_{\delta}(\xb)$ is the ball around $\xb$ with radius $\delta, \stackrel{A}{x \rightarrow \bar{x}}$ means convergence within a set $A$ and for a multifunction $\Phi, \gph \Phi := \{ (x,y)| y\in\Phi(x)\}$ stands for its graph. Finally,
$\norm{(A\,\vdots\,B)}_F$ signifies the Frobenius norm of the matrix, composed horizontally from the blocks $A,B$.
\section{Preliminaries}
Throughout the whole paper, we will frequently use  the following basic notions of modern
variational analysis.
 \begin{definition}\label{DefVarGeom}
 Let $A$  be a closed set in $\mathbb{R}^{n}$ and $\bar{x} \in A$. Then
\begin{enumerate}
 \item [(i)]
 $T_{A}(\bar{x}):=\Limsup\limits_{t\searrow 0} \frac{A-\bar{x}}{t}$
 is the {\em tangent (contingent, Bouligand) cone} to $A$ at $\bar{x}$  and \\
 $ \widehat{N}_{A}(\bar{x}):=(T_{A}(\bar{x}))^{\circ} $
 is the {\em regular (Fr\'{e}chet) normal cone} to $A$ at $\bar{x}$.
 \item [(ii)]
 $ N_{A}(\bar{x}):=\Limsup\limits_{\stackrel{A}{x \rightarrow \bar{x}}} \widehat{N}_{A}(x)$
 is the {\em limiting (Mordukhovich) normal cone} to $A$ at $\bar{x}$ and, given a direction $d
 \in\mathbb{R}^{n}$,
$ N_{A}(\bar{x};d):= \Limsup\limits_{\stackrel{t\searrow 0}{d^{\prime}\rightarrow
 d}}\widehat{N}_{A}(\bar{x}+ td^{\prime})$
 is the {\em directional limiting normal cone} to $A$ at $\bar{x}$ {\em in direction} $d$.
 \end{enumerate}
\end{definition}
In this definition "Limsup" stands for the Painlev\' e-Kuratowski {\em outer set limit}.
If $A$ is convex, then $\widehat{N}_{A}(\bar{x})= N_{A}(\bar{x})$ amounts to the classical normal cone in
the sense of convex analysis and we will  write $N_{A}(\bar{x})$. By the definition, the limiting normal
cone coincides with the directional limiting normal cone in direction $0$, i.e.,
$N_A(\bar{x})=N_A(\bar{x};0)$, and $N_A(\bar{x};d)=\emptyset$ whenever $d\not\in T_A(\bar{x})$.

The above listed cones enable us to describe the local behavior of set-valued maps via various
generalized derivatives. Consider a closed-graph multifunction $F:\R^n\tto\R^m$ and the point
$(\xb,\yb)\in \gph F$.

\begin{definition}\label{DefGenDeriv}
\begin{enumerate}

\item[(i)]
 The multifunction $\widehat D^\ast F(\xb,\yb ):
 \mathbb{R}^{m}\rightrightarrows\mathbb{R}^{n}$, defined by
\[
\widehat D^\ast F(\xb,\yb )(v^\ast):=\{u^\ast\in \mathbb{R}^{n} | (u^\ast,- v^\ast)\in \widehat
N_{\gph F}(\xb,\yb )\}, v^\ast\in \mathbb{R}^{m}
\]
is called the {\em regular (Fr\'echet) coderivative} of $F$ at $(\xb,\yb )$.
\item [(ii)]
 The multifunction $D^\ast F(\xb,\yb ): \mathbb{R}^{m}\rightrightarrows\mathbb{R}^{n}$,
 defined by
\[
D^\ast F(\xb,\yb )(v^\ast):=\{u^\ast\in \mathbb{R}^{n} | (u^\ast,- v^\ast)\in N_{\gph F}(\xb,\yb )\},
v^\ast\in \mathbb{R}^{m}
\]
is called the {\em limiting (Mordukhovich) coderivative} of $F$ at $(\xb,\yb )$.
\item [(iii)]
 Given a pair of directions $(u,v) \in \mathbb{R}^{n} \times \mathbb{R}^{m}$, the
 multifunction
 $D^\ast F((\xb,\yb ); (u,v)):
 \mathbb{R}^{n}\rightrightarrows\mathbb{R}^{m}$, defined by
\begin{equation*}
D^\ast  F((\xb,\yb ); (u,v))(v^\ast):=\{u^\ast \in \mathbb{R}^{n} | (u^\ast,-v^\ast)\in N_{\gph
F}((\xb,\yb ); (u,v)) \}, v^\ast\in \mathbb{R}^{m}
\end{equation*}
is called the {\em directional limiting coderivative} of $F$ at $(\xb,\yb )$ in direction $(u,v)$.
\end{enumerate}
\end{definition}

For the properties of the cones $T_A(\bar{x})$, $\widehat N_A(\bar{x})$ and $N_A(\bar{x})$ from
Definition
\ref{DefVarGeom} and generalized derivatives (i) and (ii) from Definition \ref{DefGenDeriv}
we refer the interested reader to the monographs \cite{RoWe98} and \cite{Mo18}. The directional limiting
normal cone and coderivative were introduced by the first author in \cite{Gfr13a} and various properties
of these objects can be found also in \cite{GO3} and the references therein. Note that $D^\ast  F(\xb,\yb
)=D^\ast  F((\xb,\yb ); (0,0))$ and that $\dom D^\ast  F((\xb,\yb ); (u,v))=\emptyset$ whenever $v\not\in
DF(\xb,\yb)(u)$.

Recall that a set-valued mapping $F:\R^n\tto\R^m$ is said to be {\em metrically regular} around a point $(\xb,\yb)\in\gph F$, if the graph of $F$ is locally closed at $(\xb,\yb)$, and there is a constant $\kappa\geq 0$ along with neighborhoods of $U$ of $\xb$ and $V$  of $\yb$ such that
\[\dist{x,F^{-1}(y)}\leq \kappa\, \dist{y,F(x)} \mbox{ for all } (x,y)\in U\times V.\]
The infimum of $\kappa$ over all such combinations of $\kappa$, $U$ and $V$ is called the {\em regularity modulus}  for $F$ at $(\xb,\yb)$  and denoted by ${\rm reg\;}F(\xb,\yb)$.
The following statement follows from \cite[Theorem 9.43]{RoWe98}
\begin{theorem}\label{ThMetrReg}A mapping $F:\R^n\tto\R^m$ is metrically regular at $(\xb,\yb)\in\gph F$, if and only if $\gph F$ is locally closed at $(\xb,\yb)$ and
\begin{equation}
  \label{EqMoCrit} 0\in D^*F(\xb,\yb)(y^*)\ \Rightarrow\ y^*=0.
\end{equation}
Further, in this case one has
\begin{equation}
  \label{EqModMetrReg} {\rm reg\;}F(\xb,\yb)=1/\min\{\dist{0,D^*F(\xb,\yb)(y^*)}\mv\norm{y^*}=1\}.
\end{equation}
\end{theorem}
In the construction of the announced globally convergent hybrid algorithms we employ the notion of monotonicity.
\begin{definition}\label{DefMonoton}A mapping $F:\R^n\tto\R^n$ is said to be {\em monotone} if it has the property that
\[\skalp{y_2-y_1,x_2-x_1}\geq 0\ \mbox{ for all }(x_1,y_1),(x_2,y_2)\in\gph F.\]
If, in addition, there is some $\mu>0$ such that
\[\skalp{y_2-y_1,x_2-x_1}\geq \mu\norm{x_2-x_1}^2\ \mbox{ for all }(x_1,y_1),(x_2,y_2)\in\gph F,\]
the mapping is called {\em strongly monotone}.
\end{definition}

Recall that a monotone mapping $F:\R^n\tto\R^n$
is maximal monotone if no enlargement of its graph is possible in $\R^n\times\R^n$
without destroying monotonicity.
Given a maximal monotone mapping $F$ and a positive real $\lambda$, the mapping $(I+\lambda F)^{-1}$ is called the {\em resolvent} of $F$. It is well known that this mapping is single-valued and Lipschitz on the whole $\R^n$, see, e.g., \cite[Theorem 12.12]{RoWe98}.

\section{On the semismooth* Newton method}
In this section we describe the semismooth* Newton method as introduced in \cite{GfrOut19a}. Consider the inclusion
\begin{equation}
  \label{EqIncl}0\in F(x),
\end{equation}
where $F:\R^n\tto\R^m$ is a set-valued mapping with closed graph. The semismoothness* property of $F$ can be defined as follows.
\begin{definition}
  A set-valued mapping $F:\R^n\tto\R^m$ is called {\em \ssstar} at a point $(\xb,\yb)\in\gph F$, if
 for all $(u,v)\in\R^n\times\R^m$ we have
 \begin{equation}\label{EqSemiSmooth}
\skalp{u^*,u}=\skalp{v^*,v}\ \forall (v^*,u^*)\in\gph D^*F((\xb,\yb);(u,v)).
\end{equation}
\end{definition}
In some situations it is convenient to make use of equivalent
characterizations in terms of standard (regular and limiting)  coderivatives, respectively.

\begin{proposition}[{\cite[Corollary 3.3]{GfrOut19a}}]\label{PropCharSemiSmooth}Let $F:\R^n\tto\R^m$ and $(\xb,\yb)\in \gph F$ be given.
Then the following three statements are equivalent.
\begin{enumerate}
\item[(i)] $F$ is \ssstar at $(\xb,\yb)$.
\item[(ii)] For every $\epsilon>0$ there is some $\delta>0$ such that
\begin{multline}
\label{EqCharSemiSmoothReg}
\vert \skalp{x^*,x-\xb}-\skalp{y^*,y-\yb}\vert\leq \epsilon
\norm{(x,y)-(\xb,\yb)}\norm{(x^*,y^*)} \\ \forall(x,y)\in \B_\delta(\xb,\yb)\ \forall
(y^*,x^*)\in\gph \widehat D^*F(x,y).
\end{multline}
\item[(iii)] For every $\epsilon>0$ there is some $\delta>0$ such that
  \begin{multline}\label{EqCharSemiSmoothLim}
\vert \skalp{x^*,x-\xb}-\skalp{y^*,y-\yb}\vert\leq \epsilon
\norm{(x,y)-(\xb,\yb)}\norm{(x^*,y^*)} \\ \forall(x,y)\in \B_\delta(\xb,\yb)\ \forall
(y^*,x^*)\in\gph D^*F(x,y).
\end{multline}
\end{enumerate}
\end{proposition}

The idea behind the semismooth* Newton method for solving \eqref{EqIncl} is as follows. If $F$ is \ssstar at $(\xb,0)$ and we are given some point $(x,y)\in\gph F$ close to $(\xb,0)$, then for every $(y^*,x^*)\in\gph D^*F(x,y)$ there holds
\[\skalp{x^*,x-\xb}=\skalp{y^*,y-0}+\oo(\norm{(x,y)-(\xb,\yb)}\norm{(x^*,y^*)})\]
by the definition of the semismoothness* property. We choose now $n$ pairs $(v_i^*,u_i^*)\in \gph D^*F(x,y)$, $i=1,\ldots, n$, compute a solution $\Delta x$ of the system
\begin{equation}\label{EqBasSystem}\skalp{x_i^*,\Delta x}=-\skalp{y_i^*,y},\ i=1,\ldots,n\end{equation}
and expect that $\norm{(x+\Delta x)-\xb}=\oo(\norm{x-\xb}$.

In order to work out this basic idea, we introduce the following notation. Given $(x,y)\in\gph F$, we denote by ${\cal A}F(x,y)$ the collection of all pairs of $n\times n$ matrices
$(A,B)$, such that there are $n$ elements $(y_i^*,x_i^*)\in \gph D^*F(x,y)$, $i=1,\ldots, n$, and the
$i$-th row of $A$ and $B$ are ${x_i^*}^T$ and ${y_i^*}^T$, respectively. Thus, the system \eqref{EqBasSystem} is of the form
\[A\Delta x=-By\]
with $(A,B)\in {\cal A}F(x,y)$. This system should have a unique solution and this leads us to the definition
\[{\cal A}_{\rm reg}F(x,y):=\{(A,B)\in {\cal A}F(x,y)\mv A\mbox{ non-singular}\}.\]
The set ${\cal A}_{\rm reg}F(x,y)$ is nonempty if, e.g., the mapping $F$ is strongly metrically regular around $(x,y)$, cf. \cite[Theorem 4.1]{GfrOut19a}. However, strong metric regularity is only a sufficient condition, for the problem \eqref{EqVI2ndK} we will show below that the weaker assumption of metric regularity is also sufficient.

For the local convergence analysis of the \ssstar Newton method, the following result plays a central role.
\begin{proposition}[{\cite[Proposition 4.3]{GfrOut19a}}]\label{PropConv} Assume that the mapping $F:\R^n\tto\R^n$ is \ssstar at
$(\xb,0)\in\gph F$. Then for every $\epsilon>0$ there is some $\delta>0$ such that for every $(x,y)\in
\gph F\cap \B_{\delta}(\xb,0)$ and every pair $(A,B)\in {\cal A}_{\rm reg}F(x,y)$ one has
\begin{equation}\label{EqBndNewtonStep}\norm{(x-A^{-1}By)-\xb }\leq
\epsilon\norm{A^{-1}}\norm{(A\,\vdots\, B)}_F\norm{(x,y)-(\xb,0)}.\end{equation}
\end{proposition}
As a byproduct of this statement we obtain the following corollary which is not directly related with the \ssstar Newton method.
\begin{corollary}\label{CorIsolSol}
Assume that the mapping $F:\R^n\tto\R^n$ is \ssstar at $(\xb,0)\in\gph F$ and assume that there are positive reals $\bar\delta$ and $\bar\kappa$ such that for every $(x,y)\in\gph F\cap \B_{\bar\delta}(\xb,0)$ there are matrices $(A,B)\in {\cal A}_{\rm reg}F(x,y)$ such that
\[\norm{A^{-1}}\norm{(A\,\vdots\, B)}_F\leq\kappa.\]
Then $\xb$ is an isolated solution of the inclusion $0\in F(x)$.
\end{corollary}
\begin{proof}
By contraposition. Assume that $\xb$ is not an isolated solution and find $0<\delta <\bar \delta$ such that \eqref{EqBndNewtonStep} holds with $\epsilon=1/(2\kappa)$. Since $\xb$ is not an isolated solution, there exists another solution $\tilde x \not=\xb$ in $\B_{\delta}(\xb)$ and by picking suitable matrices $(A,B)\in {\cal A}_{\rm reg}F(\tilde x,0)$ we obtain the contradiction
  \[\norm{\tilde x-\xb}=\norm{(\tilde x-A^{-1}B0)-\xb}\leq \frac 1{2\kappa}\norm{A^{-1}}\norm{(A\,\vdots\, B)}_F\norm{(\tilde x,0)-(\xb,0)}\leq \frac 12\norm{\tilde x-\xb}.\]
\end{proof}
We are now in the position to describe the iteration step of the \ssstar Newton method. Assume we are given some iterate $x^{(k)}$. We cannot expect in general that $F(x^{(k)})\not=\emptyset$ or that $0$ is close to
$F(x^{(k)})$, even if $x^{(k)}$ is close to a solution $\xb$. Thus we perform first some  step which
yields $(\hat x^{(k)},\hat y^{(k)})\in\gph F$ as  an approximate projection of $(x^{(k)},0)$ on $\gph F$.
Further we require that ${\cal A}_{\rm reg}F(\hat x^{(k)},\hat y^{(k)})\not=\emptyset$ and  compute the
new iterate as $x^{(k+1)}=\hat x^{(k)}-A^{-1}B\hat y^{(k)}$ for some $(A,B)\in {\cal A}_{\rm reg}F(\hat
x^{(k)},\hat y^{(k)})$. This leads to    the following conceptual algorithm.

\begin{algorithm}[\ssstar Newton-type method for generalized equations]\label{AlgNewton}\mbox{ }\\
 1. Choose a starting point $x^{(0)}$, set the iteration counter $k:=0$.\\
 2. If ~ $0\in F(x^{(k)})$, stop the algorithm.\\
  3. {\bf Approximation step: } Compute
  $$(\hat x^{(k)},\hat y^{(k)})\in\gph F$$ close to $(x^{(k)},0)$ such that ${\cal
  A}_{\rm reg}F(\hat x^{(k)},\hat y^{(k)})\not=\emptyset$.\\
  4. {\bf Newton step: }Select
  $$(A,B)\in {\cal A}_{\rm reg}F(\hat x^{(k)},\hat y^{(k)})$$ and compute the new iterate
  $$x^{(k+1)}=\hat x^{(k)}-A^{-1}B\hat y^{(k)}.$$\\
  5. Set $k:=k+1$ and go to 2.
\end{algorithm}

Now let us consider convergence properties of Algorithm \ref{AlgNewton}. Given two reals $L,\kappa>0$ and a solution $\xb$ of \eqref{EqIncl}, we denote
\[\G_{F,\xb}^{L,\kappa}(x):=\{(\hat x,\hat y,A,B)\mv \norm{(\hat x-\xb,\hat y)}\leq L\norm{x-\xb},\
(A,B)\in {\cal A}_{\rm reg}F(\hat x,\hat y), \norm{A^{-1}}\norm{(A\,\vdots\,B)}_F\leq\kappa\}.\]
\begin{theorem}[{\cite[Theorem 4.4]{GfrOut19a}}]\label{ThConvSemiSmmooth1}
  Assume that $F$ is \ssstar at $(\xb,0)\in\gph F$ and assume that there are  $L,\kappa>0$ such
  that for every $x\not\in F^{-1}(0)$ sufficiently close to $\xb$ we have
  $\G_{F,\xb}^{L,\kappa}(x)\not=\emptyset$. Then there exists some $\delta>0$ such that for every
  starting point $x^{(0)}\in\B_\delta(\xb)$ Algorithm \ref{AlgNewton} either stops after
  finitely many iterations at a solution or produces a sequence $x^{(k)}$ which converges
  superlinearly to $\xb$, provided we choose in every iteration $(\hat x^{(k)},\hat y^{(k)},A,B)\in
  \G_{F,\xb}^{L,\kappa}(x^{(k)})$.
\end{theorem}
According to Theorem \ref{ThConvSemiSmmooth1}, the outcome $(\hat x^{(k)},\hat y^{(k)})\in\gph F$ from the approximation step has to fulfill the inequality
\begin{equation}\label{EqBndApprStep}\norm{(\hat x^{(k)},\hat y^{(k)})-(\xb,0)}\leq L\norm{x^{(k)}-\xb}.\end{equation}
It is easy to show  that this estimate holds true if
\[\norm{(\hat x^{(k)},\hat y^{(k)})-(x^{(k)},0))}\leq \beta \, \dist{(x^{(k)},0),\gph F},\]
i.e., $(\hat x^{(k)},\hat y^{(k)})$ is some approximate projection of $(x^{(k)},0)$ on $\gph F$. In fact, it suffices when the deviation of $(\hat x^{(k)},\hat y^{(k)})$ from the exact  projection is proportional to the distance $\dist{(x^{(k)},0),\gph F}$. So the approximation of the projection can be rather crude.\\

In the computation of matrices A, B needed in the Newton step we will make use of the following result which is interesting also for its own sake.

\begin{theorem}\label{ThCoderSubdiff}
  Let $q:\R^n\to\bar\R$ be proper convex and lsc. Then for every $(x,x^*)\in \gph\partial q$ there is a positive semidefinite matrix $G$ with $\norm{G}\leq 1$ such that
  \begin{equation}\label{EqRgeG}\rge{G}:=\{\big((I-G)v^*,G v^*\big)\mv v^*\in\R^n\}\subset \gph D^*(\partial q)(x,x^*).\end{equation}
\end{theorem}
\begin{proof}
  Consider  the Moreau envelope function
\begin{equation} e_1 q(y):=\inf_x\big(q(x)+\frac 1{2}\norm{x-y}^2\big).\end{equation}
By \cite[Exercise 12.23]{RoWe98}, $e_1q$ is continuously differentiable  on $\R^n$ and $\nabla e_1q$ is a maximal monotone, single valued mapping, which is Lipschitz continuous with constant $1$, and
\[\nabla e_1q=(I+(\partial q)^{-1})^{-1}.\]
Thus
\[y^*=\nabla e_1q(y)\ \Leftrightarrow\ y\in (I+(\partial q)^{-1})(y^*)\ \Leftrightarrow\ y^*\in \partial q(y-y^*)\]
and, consequently, $(z,z^*)\in \gph \partial q \Leftrightarrow z^*=\nabla e_1q(z+z^*)$.
Next, consider an element $G$ from the B-subdifferential $\overline{\nabla}(\nabla e_1q)(x+x^*)$ together with sequences $y_k\to x+x^*$ and $G_k\to G$ with $G_k=\nabla(\nabla e_1q)(y_k)$. From the monotonicity of $\nabla e_1q$ it follows that $G_k$ is positive semidefinite  and hence so is $G$ as well. Further, by \cite[Theorem 13.52]{RoWe98}, $G$ is symmetric and from the Lipschitz continuity of $\nabla e_1q$ we deduce that $\norm{G}\leq 1$. Since $G$ belongs to the B-subdifferential of $\nabla e_1q$ at $x+x^*$, we have
\[(v^*, G^Tv^*)=(v^*, Gv^*)\in D^*(\nabla e_1q)(x+x^*,x^*)=D^*\Big((I+(\partial q)^{-1})^{-1}\Big)(x+x^*,x^*)\ \forall v^*\in\R^n.\]
Taking into account some elementary calculus rules for coderivatives, we conclude that
\begin{align*}
  &(v^*, Gv^*)\in D^*\Big((I+(\partial q)^{-1})^{-1}\Big)(x+x^*,x^*)\Leftrightarrow(-Gv^*,-v^*)\in D^*(I+(\partial q)^{-1})(x^*, x+x^*)\\
  &\Leftrightarrow
  (-Gv^*,-v^*+Gv^*)\in D^*(\partial q)^{-1}(x^*, x)\Leftrightarrow (v^*-Gv^*,Gv^*)\in D^*(\partial q)(x, x^*),
\end{align*}
and the assertion of the theorem follows.
\end{proof}
\section{Implementation of the \ssstar Newton method\label{SecTwoImpl}}
There are a lot of possibilities how to implement the \ssstar Newton method. Apart from the Newton step, which is not uniquely determined by different selections of the coderivatives, there is a multitude of possibilities how to perform the approximation step. In this section we will construct  an implementable version of the \ssstar Newton method for the numerical solution of GE \eqref{EqVI2ndK}. We restrict ourselves to the case where the approximation step is performed by means of the mapping $u_\gamma$ defined as
\begin{equation}\label{Eq_u_gamma}u_\gamma(x):=\argmin_u(\frac 12 \gamma\norm{u}^2+\skalp{f(x),u}+q(x+u)),\end{equation}
where $\gamma>0$ is some scaling parameter. Note that
$u_\gamma $ is clearly single-valued due to the strong convexity of the objective.
The first-order (necessary and sufficient) optimality condition reads as
\begin{equation}\label{EqOptCond_u}0\in \gamma u_\gamma(x)+f(x)+\partial q(x+u_\gamma(x)),\end{equation}
which can be equivalently written as
\[\gamma x-f(x)\in (\gamma I+\partial q)(x+u_\gamma(x)).\]
Let us premultiply this inclusion by $\lambda:=1/\gamma$.  One obtains that
\[x- \lambda f(x)\in ( I+\lambda\partial q)(x+ u_\gamma(x)),\]
which yields the equality
\begin{equation}\label{EqFB}x+ u_\gamma(x) = (I+\lambda\partial q)^{-1}(x- \lambda f(x)),\end{equation}
because the resolvent $(I+\lambda\partial q)^{-1}$ is single-valued due the maximal monotonicity of $\partial q$. Since this resolvent is also nonexpansive, cf. \cite[Theorem 12.12]{RoWe98}, for arbitrary two points $x,x'\in\R^n$ we obtain the bounds
\begin{align}
\label{EqLip1}  &\norm{(x+u_\gamma(x))-(x'+u_\gamma(x')}\leq \norm{(x-x')-\lambda(f(x)-f(x'))}\leq \norm{x-x'}+\frac 1\gamma\norm{f(x)-f(x')}\\
\label{EqLip2}  &\norm{u_\gamma(x)-u_\gamma(x')}\leq  2\norm{x-x'}+\frac 1\gamma\norm{f(x)-f(x')}.
\end{align}
They will be used in the estimates below.
\begin{remark}
Equation \eqref{EqFB} tells us, that $x+u_\gamma(x)$ is the outcome of one step of the so-called {\em forward-backward splitting method}, see, e.g., \cite{LiMe79}.
\end{remark}

Our approach is based on an equivalent reformulation of \eqref{EqVI2ndK} in form of the GE
\begin{equation}\label{EqVI-alt}0\in\F(x,d):=\myvec{f(x)+\partial q(d)\\x-d}\end{equation}
in variables $(x,d)\in\R^n\times \R^n$. Clearly, $\xb$ is a solution of \eqref{EqVI2ndK} if and only if $(\xb,\xb)$ is a solution of \eqref{EqVI-alt}. Further, it is easy to see that
$\F$ is \ssstar at $((\xb,\xb),(0,0))$ if and only if $\partial q$ is \ssstar at $(\xb,-f(\xb))$.

We start with the description of the approximation step. Given $(x^{(k)},d^{(k)})$ and a scaling parameter $\gamma^{(k)}$, we compute $u^{(k)}:=u_{\gamma^{(k)}}(x^{(k)})$  and  set
\begin{equation}
  \label{EqResApprStep1}\hat x^{(k)}= x^{(k)},\ \hat d^{(k)}=x^{(k)}+u^{(k)}\quad \mbox{and}\quad \hat y^{(k)}=(\hat y_1^{(k)},\hat y_2^{(k)})=(-\gamma^{(k)} u^{(k)},u^{(k)}).
\end{equation}
We observe that
\[((\hat x^{(k)},\hat d^{(k)}),(\hat y_1^{(k)},\hat y_2^{(k)}))\in \gph \F,\]
which follows immediately from the first-order optimality condition \eqref{EqOptCond_u}. Note that the outcome of the approximation step does not depend on the auxiliary variable $d^{(k)}$.  In order to apply Theorem \ref{ThConvSemiSmmooth1}, we shall show the existence of a real $L>0$ such that the estimate
\begin{equation}\label{EqEstApprStep1}
  \norm{((\hat x^{(k)}-\xb, \hat d^{(k)}-\xb), \hat y^{(k)}}\leq L\norm{(x^{(k)}-\xb,d^{(k)}-\xb)},
\end{equation}
corresponding to \eqref{EqBndApprStep}, holds for all $(x^{(k)},d^{(k)})$ with $x^{(k)}$ close to $\xb$. We observe that the left-hand side of \eqref{EqEstApprStep1} amounts to
\begin{align}\nonumber
  \norm{((\hat x^{(k)}-\xb, \hat x^{(k)}+ u^{(k)}-\xb),(-\gamma^{(k)} u^{(k)}, u^{(k)})}&\leq \norm{(\hat x^{(k)}-\xb, \hat x^{(k)}-\xb,0,0)}+\norm{(0, u^{(k)},-\gamma^{(k)} u^{(k)}, u^{(k)})}\\
  \label{EqAuxBnd1}&\leq  2\norm{\hat x^{(k)}-\xb}+(2+\gamma^{(k)})\norm{u^{(k)}}.
\end{align}
Since $u_{\gamma^{(k)}}(\xb)=0$, we obtain from \eqref{EqLip2} the bounds
\begin{align}\label{EqAppr1Lip1}&\norm{\hat d^{(k)}-\xb}\leq \norm{x^{(k)}-\xb}+\frac 1{\gamma^{(k)}}\norm{f(x^{(k)})-f(\xb)}\\
\label{EqAppr1Lip2}&\norm{u^{(k)}}\leq 2\norm{x^{(k)}-\xb}+\frac 1{\gamma^{(k)}}\norm{f(x^{(k)})-f(\xb)}.\end{align}
The latter estimate, together with \eqref{EqAuxBnd1}, imply
\begin{align}\nonumber\norm{((\hat x^{(k)}-\xb, \hat d^{(k)}-\xb), \hat y^{(k)}}&\leq \Big(2+(2+\gamma^{(k)})\big(2+ \frac l{\gamma^{(k)}}\big)\Big)\norm{x^{(k)}-\xb}\\
\label{EqAppr1L}&\leq  \Big(2+(2+\gamma^{(k)})\big(2+ \frac l{\gamma^{(k)}}\big)\Big)\norm{(x^{(k)}-\xb,d^{(k)}-\xb)},\end{align}
where $l$ is the Lipschitz constant of $f$ on a neighborhood of $\xb$. Thus the desired inequality \eqref{EqEstApprStep1} holds, as long as $\gamma^{(k)}$ remains bounded and bounded away from $0$.

Let us now consider the Newton step. By calculus of coderivatives we have for any $s,s^{*} \in \mathbb{R}^{n}$ the  equality
\[D^*\F((\hat x^{(k)},\hat d^{(k)}),\hat y^{(k)})\myvec{s\\s^*}=\myvec{\nabla f(\hat x^{(k)})^Ts+s^*\\D^*(\partial q)(\hat d^{(k)}, \hat d^*{}^{(k)})(s)-s^*},\]
where $\hat d^*{}^{(k)}:=\hat y_1^{(k)}-f(\hat x^{(k)})\in\partial q(\hat d^{(k)})$. Assume that, according to Theorem \ref{ThCoderSubdiff}, we have  a symmetric, positive definite matrix $G^{(k)}$ satisfying $\norm{G^{(k)}}\leq 1$ and
\begin{equation}\label{Eq_G_k}\rge{G^{(k)}}\subseteq \gph D^*(\partial q)(\hat d^{k}, \hat d^*{}^{(k)})\end{equation}
at our disposal. We now choose
\begin{align*}&v_i^*:=\myvec{(I-G^{(k)})e_i\\0},\ u_i^*=\myvec{\nabla f(\hat x^{(k)})^T(I-G^{(k)})e_i\\G^{(k)}e_i},\ i=1,\ldots,n,\\ &v_i^*:=\myvec{0\\e_{i-n}},\ u_i^*=\myvec{e_{i-n}\\-e_{i-n}},\ i=n+1,\ldots,2n,
\end{align*}
so that $(A,B)\in{\cal A}\F((\hat x^{(k)},\hat d^{(k)}),\hat y^{(k)}),$
where
\begin{equation}\label{EQAB_Appr1}A=\left(\begin{array}{cc}(I-G^{(k)})\nabla f(\hat x^{(k)})&G^{(k)}\\I&-I\end{array}\right),\ B=\left(\begin{array}{cc}I-G^{(k)})&0\\0&I\end{array}\right).\end{equation}
Elementary calculations show that
\[A^{-1}=\left(\begin{array}{cc}C^{-1}&G^{(k)}C^{-1}\\C^{-1}&-(I-G^{(k)})\nabla f(\hat x^{(k)})C^{-1}\end{array}\right),\]
provided  the matrix $C:=(I-G^{(k)})\nabla f(\hat x^{(k)})+G^{(k)}$ is nonsingular.  In this case, since $G^{(k)}$ is positive semidefinite and $\norm{G^{(k)}}\leq 1$,  the matrices $A,B$ given by \eqref{EQAB_Appr1} fulfill a bound of the form
\begin{equation}\label{EqBndAB}\norm{A^{-1}}\norm{(A\vdots B)}_F\leq \norm{\big((I-G^{(k)})\nabla f(\hat x^{(k)})+G^{(k)}\big)^{-1}}\big(C_1+C_2\norm{\nabla f(\hat x^{(k)})}\big)^2\end{equation}
with constants $C_1,C_2>0$.

\begin{proposition}\label{PropAinv}
  Assume that
  \begin{equation}\label{EqCondExistNewtonDir}0\in \nabla f(\hat x^{(k)})^Ts +D^*(\partial q)(\hat d^{k}, \hat d^*{}^{(k)})(s)\ \Rightarrow s=0\end{equation}
  Then  $(I-G^{(k)})\nabla f(\hat x^{(k)})+G^{(k)}$ is non-singular and
  \[\norm{((I-G^{(k)})\nabla f(\hat x^{(k)})+G^{(k)})^{-1}}\leq 1+ \frac 1\mu \norm{\nabla f(\hat x^{(k)})-I},\]
  where
  \begin{equation}\label{EqMu}\mu =\min_{\norm{s}=1}\dist{0,\nabla f(\hat x^{(k)})^Ts +D^*(\partial q)(\hat d^{k}, \hat d^*{}^{(k)})(s)}.\end{equation}
\end{proposition}
\begin{proof}
  By the definition of $\mu$ we have for every $u\in\R^n$ the estimate
  \begin{align*}\norm{\nabla f(\hat x^{(k)})^T(I-G^{(k)})u +G^{(k)}u} &\geq  \dist{0,\nabla f(\hat x^{(k)})^T(I-G^{(k)})u +D^*(\partial q)(\hat d^{k}, \hat d^*{}^{(k)})((I-G^{(k)})u)}\\
  &\geq \mu\norm{(I-G^{(k)})u}\end{align*}
  implying
  \[\norm{\nabla f(\hat x^{(k)})^T(I-G^{(k)})u +G^{(k)}u}\geq \frac \mu{\mu+\norm{\nabla f(\hat x^{(k)})-I}}\norm{u}\]
  whenever $\norm{u}\leq \norm{(I-G^{(k)})u}(\mu+\norm{\nabla f(\hat x^{(k)})-I})$. On the other hand, if $\norm{u}> \norm{(I-G^{(k)})u}(\mu+\norm{\nabla f(\hat x^{(k)})-I})$, then
   \begin{align*}\norm{\nabla f(\hat x^{(k)})^T(I-G^{(k)})u +G^{(k)}u}&=\norm{u+(\nabla f(\hat x^{(k)})^T-I)(I-G^{(k)})u}\\
   &\geq \norm{u}-\norm{\nabla f(\hat x^{(k)})^T-I}\norm{(I-G^{(k)})u}>\norm{u}- \frac{\norm{\nabla f(\hat x^{(k)})^T-I}}{\mu+\norm{\nabla f(\hat x^{(k)})-I}}\norm{u}\\
   &=\frac\mu{\mu+\norm{\nabla f(\hat x^{(k)})-I}}\norm{u},\end{align*}
   where we have taken into account that $\norm{\nabla f(\hat x^{(k)})^T-I}=\norm{\nabla f(\hat x^{(k)})-I}$. Hence,
   \[\norm{(\nabla f(\hat x^{(k)})^T(I-G^{(k)}) +G^{(k)})u}\geq \frac\mu{\mu+\norm{\nabla f(\hat x^{(k)})-I}}\norm{u}\quad
   \forall u\]
   and
   \[\norm{((I-G^{(k)})\nabla f(\hat x^{(k)})+G^{(k)})^{-1}}=\norm{(\nabla f(\hat x^{(k)})^T(I-G^{(k)}) +G^{(k)})^{-1}}\leq \frac {\mu+\norm{\nabla f(\hat x^{(k)})-I)}}\mu\]
   follows.
\end{proof}

In order to actually perform the Newton step, we denote by $(\Delta x^{(k)}, \Delta d^{(k)})$ the solution of  the linear system
\[A\myvec{\Delta x\\\Delta d}=
\left(\begin{array}{cc}(I-G^{(k)})\nabla f(\hat x^{(k)})&G^{(k)}\\I&-I\end{array}\right)\myvec{\Delta x\\\Delta d}
=-B\hat y^{(k)}=\myvec{\gamma^{(k)}(I-G^{(k)})u^{(k)}\\ u^{(k)}}\]
and set $x^{(k+1)}:=\hat x^{(k)}+\Delta x^{(k)}$.
Note that the variable $\Delta d$  can be easily eliminated from the system above yielding
\begin{equation}\label{EqNewtonStepAppr1}
\begin{split}
&((I-G^{(k)})\nabla f(\hat x^{(k)})+G^{(k)})\Delta x^{(k)}= (\gamma^{(k)}(I-G^{(k)})+G^{(k)})u^{(k)},  \\
&x^{(k+1)}=d^{(k+1)}=x^{(k)}+\Delta x^{(k)}.
\end{split}
\end{equation}

We now present sufficient conditions  for the fulfilment of condition \eqref{EqCondExistNewtonDir}. Note that by Theorem \ref{ThMetrReg}, condition \eqref{EqCondExistNewtonDir} is fulfilled if and only if the mapping $H_{u^{(k)}}:\R^n\tto\R^n$, defined by  $H_{u^{(k)}}(x):=f(x)+\partial q(x+u^{(k)})$, is metrically regular around $(x^{(k)},\hat y_1^{(k)})$.
\begin{lemma}\label{LemBndInvMetrReg}Assume that $H$ is metrically regular around $(\xb,0)\in \gph H$ and  let two positive real numbers $\underline\gamma\leq \bar\gamma$ be given.
 Then for every $\kappa'>{\rm reg\;} H(\xb,0)$ there exists some positive radius $\rho'$ such that for every $x^{(k)}\in\B_{\rho'}(\xb)$ and every $\gamma^{(k)}\in[\underline{\gamma},\bar\gamma]$ one has
\[\dist{0, \nabla f(\hat x^{(k)})^Ts+D^*(\partial q)(\hat d^{(k)},\hat d^*{}^{(k)})(s)}\geq \Big(\frac 1{\kappa'}-\norm{\nabla f(\hat x^{(k)})-\nabla f(\hat d^{(k)})}\Big)\norm{s}
  \geq \frac 1{2\kappa'}\norm{s}\ \forall s\]
and consequently
\[\norm{((I-G^{(k)})\nabla f(\hat x^{(k)})+G^{(k)})^{-1}}\leq 1+ 2\kappa'\norm{\nabla f(\hat x^{(k)})-I}\]
and
\[\norm{\Delta x^{(k)}}\leq \Big(1+2\kappa'\norm{\nabla f(\hat x^{(k)})-I}\Big)\max\{1,\gamma^{(k)}\}\norm{u^{(k)}}.\]
\end{lemma}
\begin{proof}Let $\kappa'>{\rm reg\;} H(\xb,0)$ be arbitrarily fixed. We can find  a positive radius $\rho>0$ such that
  \[\dist{x, H^{-1}(y)}\leq \kappa' \dist{y,H(x)}\ \forall (x,y)\in\B_\rho(\xb)\times \B_\rho(0).\]
  Hence, by Theorem \ref{ThMetrReg}, for every $(x,y)\in\gph H\cap (\inn \B_\rho(\xb)\times \inn \B_\rho(0))$ we have
  \[\dist{0, \nabla f(x)^Ts+D^*(\partial q)(x,y-f(x))(s)}\geq \frac 1{\kappa'}\norm{s}\ \forall s.\]
  We can choose $\rho$ small enough such that
  \[\norm{\nabla f(x)-\nabla f(d)}\leq \frac 1{2\kappa'}\ \forall x,d\in\B_\rho(\xb).\]
  Let $l$ denote the Lipschitz constant of $f$ on $\B_\rho(\xb)$ and choose $\rho'>0$ such that
  \[(\gamma+l)(2+\frac l\gamma)\rho'<\rho,\quad (1+\frac l\gamma)\rho'<\rho\quad \forall \gamma\in[\underline{\gamma},\bar\gamma].\]
  Consider $(x^{(k)},d^{(k)})\in \B_{\rho'}(\xb)\times\R^n$ and $\gamma^{(k)}\in[\underline{\gamma},\bar\gamma]$. By \eqref{EqAppr1Lip1} we have $\norm{\hat d^{(k)}-\xb}\leq (1+\frac l{\gamma^{(k)}})\rho'<\rho$. Further,
  $f(\hat d^{(k)})+\hat d^*{}^{(k)}=f(\hat d^{(k)})-\gamma^{(k)} u^{(k)}-f(\hat x^{(k)})\in H(\hat d^{(k)})$ and
  \[\norm{f(\hat d^{(k)})+\hat d^*{}^{(k)}}\leq (l+\gamma^{(k)})\norm{u^{(k)}}\leq (l+\gamma^{(k)})(2+\frac l{\gamma^{(k)}})\rho'<\rho,\]
  where we have used \eqref{EqAppr1Lip2}. Thus
 \[\dist{0, \nabla f(\hat d^{(k)})^Ts+D^*(\partial q)(\hat d^{(k)},\hat d^*{}^{(k)})(s)}\geq \frac 1{\kappa'}\norm{s}\quad \forall s\]
  implying
  \[\dist{0, \nabla f(\hat x^{(k)})^Ts+D^*(\partial q)(\hat d^{(k)},\hat d^*{}^{(k)})(s)}\geq \Big(\frac 1{\kappa'}-\norm{\nabla f(\hat x^{(k)})-\nabla f(\hat d^{(k)})}\Big)\norm{s}
  \geq \frac 1{2\kappa'}\norm{s}\quad \forall s.\]
  Hence we can apply Proposition \ref{PropAinv} to obtain
  \[\norm{((I-G^{(k)})\nabla f(\hat x^{(k)})+G^{(k)})^{-1}}\leq 1+ 2\kappa'\norm{\nabla f(\hat x^{(k)})-I}.\]
  The estimate for $\norm{\Delta x^{(k)}}$ follows from \eqref{EqNewtonStepAppr1} by taking into account that
 $\norm{(\gamma^{(k)}(I-G^{(k)})+G^{(k)})u^{(k)}}\leq \max\{1,\gamma^{(k)}\}\norm{u^{(k)}}$ because $G^{(k)}$ is symmetric and positive semidefinite with $\norm{G^{(k)}}\leq 1$.
\end{proof}

Under an additional condition we can give an estimate for the constant $\mu$ defined by \eqref{EqMu} and for the length of the Newton direction $\norm{\Delta x^{(k)}}$.
\begin{lemma}\label{LemBndInv}
  Assume that $f$ is monotone. Given  $x\in\R^n$ and $d\in\dom\partial q$, consider the numbers
  \begin{align}\label{EqMu_f}&\mu_f(x):=\min\{\skalp{\nabla f(x)u,u}\mv \norm{u}=1\},\\
  \label{EqMu_q} &\mu_q(d):=\lim_{\rho \downarrow 0}\inf\{\frac{\skalp{d_1^*-d_2^*, d_1-d_2}}{\norm{d_1-d_2}^2}\mv d_i\in\B_\rho(d), (d_i,d_i^*)\in\gph\partial q,\ i=1,2, d_1\not=d_2\}.\end{align}
  Then for every $d^*\in \partial q(d)$ we have
  \[\min_{\norm{s}=1}\dist{0,\nabla f(x)^Ts +D^*(\partial q)( d, d^*)(s)}\geq \mu_f(x)+\mu_q(d).\]
 \end{lemma}
 \begin{proof}
   If $\mu_f(x)+\mu_q(d)=0$ the assertion trivially holds true. Hence we may assume $\mu_f(x)+\mu_q(d)>0$. Consider $0<\epsilon< \mu_f+\mu_q$ and pick some $\rho>0$ such that
   \begin{align*}&\inf_{x'\in\B_\rho(x)} \min\{\skalp{\nabla f(x')u,u}\mv \norm{u}=1\}\geq \mu_f(x)-\frac\epsilon2,\\
   &\inf\{\frac{\skalp{d_1^*-d_2^*, d_1-d_2}}{\norm{d_1-d_2}^2}\mv d_i\in\B_\rho(d), (d_i,d_i^*)\in\gph\partial q,\ i=1,2,\ d_1\not=d_2\}\geq \mu_q(d)-\frac \epsilon2.
   \end{align*}
   Utilizing \cite[Exercise 12.45]{RoWe98} we see that the mapping $H^\rho_{d-x}:\R^n\tto \R^n$, defined by $H^\rho_{d-x}(x'):=f(x')+\partial q(x'+(d-x))+N_{\B_\rho(x)}(x')$ is maximally monotone. Further, by construction the mapping $H^\rho_{d-x}$ is strongly monotone with constant $\mu_f(x)+\mu_g(d)-\epsilon$ and therefore its inverse is single valued and Lipschitzian on $\R^n$ with constant $1/(\mu_f(x)+\mu_g(d)-\epsilon)$. But this implies that $H^\rho_{d-x}$ is (strongly) metrically regular around every point $(x', x'^*)$ of its graph and by \cite[Theorem 9.43]{RoWe98} we obtain
   \[\min_{\norm{s}=1}\dist{0, D^*H^\rho_{d-x}(x',x'^*)(s)}\geq \mu_f(x)+\mu_g(d)-\epsilon.\]
   Taking into account that $N_{\B_\rho(x)}(x)=\{0\}$, for every $d^*\in\partial q(d)$ we have $D^*H^\rho_{d-x}(x,f(x)+d^*)(s)=\nabla f(x)^Ts+D^*(\partial q)(d,d^*)(s)$ and since we can choose $\epsilon>0$ arbitrarily small, the assertion follows.
 \end{proof}
 \begin{corollary}\label{CorExistNewtonDir}
   Assume that $f$ is monotone and assume that either $f$ or $\partial q$ is strongly monotone. Then for every iterate $x^{(k)}\in\R^n$ and every scaling parameter $\gamma^{(k)}>0$ the new iterate $x^{(k+1)}$ given by \eqref{EqResApprStep1} and \eqref{EqNewtonStepAppr1} is well defined. Moreover,
   \[\norm{\Delta x^{(k)}}\leq \Big(1+\frac1{\mu_f+\mu_q}\norm{\nabla f(\hat x^{(k)})-I}\Big)\max\{1,\gamma^{(k)}\}\norm{u^{(k)}},\]
   where $\mu_f:=\inf_{x\in\R^n} \mu_f(x)$, $\mu_q:=\inf_{d\in\dom\partial q}\mu_q(d)$.
 \end{corollary}
 \begin{proof}Note that $\mu_f = \inf_{x_1\not=x_2}\frac{\skalp{f(x_1)-f(x_2),x_1-x_2}}{\norm{x_1-x_2}^2}$ and
 \[\mu_q\geq \inf\{\frac{\skalp{d_1^*-d_2^*, d_1-d_2}}{\norm{d_1-d_2}^2}\mv (d_i,d_i^*)\in\gph\partial q,\ i=1,2,\ d_1\not=d_2\}.\]
 Thus $\mu_f+\mu_q>0$ and the assertion follows from Lemma \ref{LemBndInv}, Proposition \ref{PropAinv} and formula \eqref{EqNewtonStepAppr1}
 \end{proof}

We now prove locally superlinear convergence of the \ssstar Newton method. We restrict ourselves to the special case when   $\gamma^{(k)}$ is constant.

\begin{theorem}\label{ThAppr1Conv}
  Assume that the mapping $H=f+\partial q$  is metrically regular at $(\xb,0)\in\gph H$ and assume that $\partial q$ is \ssstar at $(\xb,-f(\xb))$. Then for every positive number $\gamma$  there exists a neighborhood $U$ of $\xb$  such that for every starting point $x^{(0)}\in U$ the  \ssstar Newton method of Algorithm \ref{AlgNewton} with $\gamma^{(k)}=\gamma$ for all $k$, with approximation step \eqref{EqResApprStep1} and with Newton step given by \eqref{EqNewtonStepAppr1}, converges superlinearly to $\xb$.
\end{theorem}
\begin{proof}
Fix $\kappa'>{\rm reg\;} H(\xb,0)$, set $\underline \gamma =\bar\gamma =\gamma$ and determine $\rho'>0$ according to Lemma \ref{LemBndInvMetrReg}. Denoting by $l$  the Lipschitz constant of $f$ on $\B_{\rho'}(\xb)$
and taking into account \eqref{EqAppr1L} and \eqref{EqBndAB}, we can conclude that
  \[\G^{L,\kappa}_{\F,(\xb,\xb)}(x,d)\not=\emptyset\ \forall (x,d)\in \B_{\rho'}(\xb)\times\R^n,\]
  where $L:=2+(2+\gamma)\big(2+ \frac l{\gamma}\big)$ and $\kappa=(1+2\kappa'(l+1))(C_1+C_2l)^2$.
  Since we choose in every step $((\hat x^{(k)},\hat d^{(k)}),\hat y^{(k)},A,B)\in \G^{L,\kappa}_{\F,(\xb,\xb)}(x^{(k)},d^{(k)})$, the assertion follows from Theorem \ref{ThConvSemiSmmooth1}.

\end{proof}

\section{Globalization}
In the last section we showed locally superlinear convergence of our implementation of the semismooth* Newton method. However, we do not only want fast local convergence but also convergence from arbitrary starting points. To this end we consider  a non-monotone line-search heuristic as well as hybrid approaches which combine this heuristic with some globally convergent method like the forward-backward (FB) splitting method, the Douglas-Rachford (DR) splitting method and some hyperplane projection method, respectively, in order to ensure global convergence.

To perform the line search we need some merit function. Similar to the damped Newton method for solving smooth equations, we use some kind of residual. Here we  define the residual by means of the approximation step, i.e., given $x$ and $\gamma>0$, we use
\begin{align}
  \label{EqRes1}
  r_\gamma(x):=\norm{(-\gamma u_\gamma(x),u_\gamma(x))}=\sqrt{1+\gamma^2}\norm{u_\gamma(x)}
\end{align}
as motivated by \eqref{EqResApprStep1}.

\subsection{A non-monotone line-search heuristic}
In general, we replace the full Newton step 4. in Algorithm \ref{AlgNewton} by a damped step of the form
\[x^{k+1}=\hat x^{(k)}+\alpha^{(k)}\triangle x^{(k)}\mbox{ with }\triangle x^{(k)}:=-A^{-1}B\hat y^{(k)}, \]
where $\alpha^{(k)}\in(0,1]$ is chosen such that the line search condition
\begin{equation}\label{EqLineSearch}r_{\gamma^{(k)}}(\hat x^{(k)}+\alpha^{(k)}s^{(k)})\leq (1+\delta^{(k)}-\mu \alpha^{(k)})r_{\gamma^{(k)}}(\hat x^{(k)})\end{equation}
is fulfilled, where $\mu\in(0,1)$ and $\delta^{(k)}$ is a given sequence of positive numbers converging to $0$.

Obviously, the step size $\alpha^{(k)}$ exists since the residual function $r_\gamma(x)$ is continuous. However, it is not guaranteed that the residual is decreasing, i.e., that $r_{\gamma^{(k)}}(x^{(k+1)})<r_{\gamma^{(k)}}(\hat x^{(k)})$.

The computation of $\alpha^{(k)}$ can be done in the usual way, e.g., we can choose the first element of a sequence $(\beta_j)$, which has  $\beta_0=1$ and converges monotonically to zero, such that the line search condition \eqref{EqLineSearch} is fulfilled.

\begin{algorithm}[Globalized \ssstar Newton heuristic for VI of the second kind]\label{AlgSSNewtHeur}\mbox{ }\\
Input: starting point $x^{(0)}$,  line search parameter $0<\nu<1$, a sequence $\delta^{(k)}\downarrow 0$, a sequence $\beta_j\downarrow 0$ with $\beta_0=1$ and a stopping tolerance $\epsilon_{tol}>0$.\\
1. Choose $\gamma^{(0)}$ and set the iteration counter $k:=0$.\\
2. If $r_{\gamma^{(k)}}(x^{(k)})\leq\epsilon_{tol}$, stop the algorithm.\\
3. Compute $G^{(k)}$ fulfilling \eqref{Eq_G_k} and the Newton direction $\Delta x^{(k)}$ by solving \eqref{EqNewtonStepAppr1}.\\
4. Determine the step size $\alpha^{(k)}$ as the first element from the sequence $\beta_j$ satisfying
\[r_{\gamma^{(k)}}(x^{(k)}+\beta_j\Delta x^{(k)})\leq (1+\delta^{(k)}-\nu \beta_j)r_{\gamma^{(k)}}(x^{(k)}).\]
5. Set $x^{(k+1)}=x^{(k)}+\alpha^{(k)}\Delta x^{(k)}$ and update $\gamma^{(k+1)}$.\\
6. Increase the iteration counter $k:=k+1$ and go to Step 2.
\end{algorithm}

Note that every evaluation of the residual function $r_\gamma(x)$ requires the computation of $u_\gamma(x)$, i.e., essentially one step of the FB splitting method. For $\gamma^{(k)}$ we suggest a choice $\gamma^{(k)}\approx\norm{\nabla f(x^{(k)}}$. Since the spectral norm $\norm{\nabla f(x^{(k)}}$ is difficult to compute, we use an easy computable norm instead, e.g., the
 maximum absolute column sum norm $\norm{\nabla f(x^{(k)})}_1$.

 Algorithm \ref{AlgSSNewtHeur} is a heuristic and we are not able to show convergence properties. Nevertheless it showed good convergence properties in practice and therefore we incorporate its principles in other algorithms to improve their performance.

 \subsection{Globally convergent hybrid approaches \label{SubSecSS_Hybrid}}

 In this subsection we suggest a combination of the \ssstar Newton method with some existing globally convergent method which exhibits both global convergence and local superlinear convergence. Assume that the used globally convergent method is formally given by some mapping $\T:\R^n\to\R^n$, which computes from some iterate $x^{(k)}$ the next iterate by
 \[x^{(k+1)}=\T(x^{(k)}).\]
 Of course, $\T$ must depend on the problem \eqref{EqVI2ndK} which we want to solve and will presumably depend also on some additional parameters which control the behavior of the method. In our notation we neglect to a large extent these dependencies.

 Consider the following well-known examples for such a mapping $\T$.
 \begin{enumerate}
   \item For the forward-backward splitting method, the mapping $\T$ is given by
   \begin{equation}\label{EqFB1}
     \T^{\rm FB}_\lambda(x)=(I+\lambda\partial q)^{-1}(I-\lambda f)(x),
   \end{equation}
   where $\lambda>0$ is a suitable prarameter. Note that $\T^{\rm FB}_\lambda(x)=x+u_{1/\lambda}(x)$.
   \item For the Douglas-Rachford splitting method we  have
   \begin{equation}\label{EqDR}
     \T^{\rm DR}_\lambda(x)=(I+\lambda f)^{-1}\Big((I+\lambda\partial q)^{-1}(I-\lambda f) +\lambda f\Big)(x)=(I+\lambda f)^{-1}(\T^{\rm FB}_\lambda+\lambda f)(x),
   \end{equation}
   where $\lambda>0$ is again some parameter.
   \item A third method is given by the hybrid projection-proximal point algorithm due to Solodov and Svaiter \cite{SolSv99}. Let $x$ and $\gamma>0$ be given and consider $\hat x=\T^{\rm FB}_{1/\gamma}(x)$, i.e. $\hat x-x=u_{\gamma}(x)$. Then $0\in \gamma(\hat x-x)+f(x)+\partial q(\hat x)$ and consequently
       \begin{equation}\label{EqPMAux}
         0\in v + \gamma(\hat x-x)+(f(x)-f(\hat x)),
       \end{equation}
       where $v:=-\gamma(\hat x-x)+f(\hat x)-f(x)\in H(\hat x).$
   Then, in the hybrid projection-proximal point algorithm the mapping $\T$ is given by the projection of $x$ on the hyperplane $\{z\mv \skalp{v,z-\hat x}=0\}$, i.e.,
   \begin{equation}\label{EqPM}\T^{\rm PM}_\gamma(x)=x-\frac{\skalp{v,\hat x-x}}{\norm{v}^2}v.\end{equation}
 \end{enumerate}
 \begin{algorithm}[Globally convergent hybrid \ssstar Newton method for VI of the second kind]\label{AlgSSNewt_Hybrid}\mbox{ }\\
Input: A method for solving \eqref{EqVI2ndK} given by the iteration operator $\T:\R^n\to\R^n$, a starting point $x^{(0)}$,  line search parameter $0<\nu<1$, a sequence $\delta^{(k)}\in (0,1)$, a sequence $\beta_j\downarrow 0$ with $\beta_0=1$  and a stopping tolerance $\epsilon_{tol}>0$.\\
1. Choose $\gamma^{(0)}$, set $r_N^{(0)}:= r_{\gamma^{(0)}}( x^{(0)})$ and set the counters $k:=0$, $l:=0$.\\
2. If $r_{\gamma^{(k)}}(x^{(k)})\leq\epsilon_{tol}$ stop the algorithm.\\
\if{3. \parbox[t]{\myAlgBox}{(optional) Apply $p$ times the operation iterator $\T$ and compute $x^{k+i}=\T(x^{(k+i-1)})$, $i=1,\ldots,p$. Set $k:=k+p$ and update $\gamma^{(k)}$.}\\}\fi
3. \parbox[t]{\myAlgBox}{Compute $G^{(k)}$ fulfilling \eqref{Eq_G_k} and the Newton direction $\Delta x^{(k)}$ by solving \eqref{EqNewtonStepAppr1}. Try to determine the step size $\alpha^{(k)}$ as the first element from the sequence $\beta_j$ satisfying $\beta_j>\delta^{(l)}$ and
\[r_{\gamma^{(k)}}(x^{(k)}+\beta_j\Delta x^{(k)})\leq (1-\nu \beta_j) r_N^{(l)}.\]}
4. If both $\Delta x^{(k)}$ and $\alpha^{(k)}$ exist, set $x^{(k+1)}=x^{(k)}+ \alpha^{(k)} \Delta x^{(k)}$, $r_N^{(l+1)} =r_{\gamma^{(k)}}(x^{(k+1)})$ and increase $l:=l+1$.\\
5. Otherwise, if the Newton direction $\Delta x^{(k)}$ or the step length  $\alpha^{(k)}$ does not exist, compute $x^{(k+1)}=\T(x^{(k)})$.\\
6. Update $\gamma^{(k+1)}$ and increase the iteration counter $k:=k+1$ and go to Step 2.
\end{algorithm}
Recall that the Newton direction $\Delta x^{(k)}$ exists, whenever condition \eqref{EqCondExistNewtonDir} is fulfilled.  In particular, by Corollary \ref{CorExistNewtonDir} this holds if $f$ is monotone and either $f$ or $\partial q$ is strongly monotone.

In what follows we denote by $k_l$ the subsequence of iterations where the new iterate $x^{k+1}$ is computed in the damped Newton Step 4, i.e.,
\[x^{(k_l)}= x^{(k_l-1)}+\alpha^{(k_l-1)}\Delta x^{(k_l-1)},\ r_N^{(l)}=r_{\gamma^{(k_l-1)}}(x^{(k_l)}).\]
\begin{theorem}\label{ThGlobConv}
  Assume that the GE \eqref{EqVI2ndK} has at least one solution and assume that the solution method given by the iteration mapping $\T:\R^n\to\R^n$ has the property that for every starting point $y^{(0)}\in\R^n$ the sequence $y^{(k)}$, given by the recursion $y^{(k+1)}=\T(y^{(k)})$, has at least one accumulation point which is a solution to the GE \eqref{EqVI2ndK}. Then for every starting point $x^{(0)}$ the sequence $x^{(k)}$ produced by Algorithm \ref{AlgSSNewt_Hybrid} with $\epsilon_{tol}=0$ and $\sum_{k=0}^\infty \delta^{(k)}=\infty$ has the following properties.
  \begin{enumerate}
  \item If the Newton iterate is accepted only finitely many times in step 4, then the sequence $x^{(k)}$ has at least one accumulation point which solves \eqref{EqVI2ndK}. Further, if the sequence $\gamma^{(k)}$ is bounded and bounded away from $0$, for every accumulation point $\xb$ of the sequence $x^{(k)}$ which is a solution to \eqref{EqVI2ndK}, the mapping $H$ is not metrically regular around $(\xb,0)$.
  \item If the Newton step is accepted infinitely many times in step 4, then every accumulation point of the subsequence $x^{(k_l)}$ is a solution to \eqref{EqVI2ndK}.
  \item If there exists an accumulation point $\xb$ of the sequence $x^{(k)}$ which solves \eqref{EqVI2ndK} and where the mapping $H$ is metrically regular and \ssstar at $(\xb,0)$, then the sequence $x^{(k)}$ converges superlinearly to $\xb$ and the Newton step in step 4 is accepted with step length $\alpha^{(k)}=1$ for all $k$ sufficiently large, provided the sequence $\gamma^{(k)}$ satisfies
      \begin{equation}\label{EqBndGamma}0<\underline{\gamma}\leq \gamma^{(k)}\leq\bar\gamma\ \forall k\end{equation}
       for some positive reals $\underline{\gamma},\bar\gamma$.
  \end{enumerate}
\end{theorem}
\begin{proof}
  The first statement is an immediate consequence of our assumption on $\T$ and the third assertion. In order to show the second statement, observe that the sequence $r_N^{(l)}$ satisfies
  $r_N^{(l+1)}\leq (1-\nu\delta^{(l)})r_N^{(l)}$ implying
  \[\lim_{l\to\infty} \ln(r_N^{(l+1)})-\ln(r_N^{(0)}) \leq\lim_{l\to\infty}\sum_{i=0}^l\ln(1-\nu\delta^{(i)})\leq -\lim_{l\to\infty}\sum_{i=0}^l\nu \delta^{(i)}=-\infty.\]
  Thus $\lim_{l\to\infty}r_N^{(l)}=\lim_{l\to\infty}\sqrt{1+{\gamma^{(k_l-1)}}^2}\norm{u_{\gamma^{(k_l-1)}}(x^{(k_l)})}=0$ and we can conclude that \[\lim_{l\to\infty}\norm{u_{\gamma^{(k_l-1)}}(x^{(k_l)})}=\lim_{l\to\infty}\gamma^{(k_l-1)}\norm{u_{\gamma^{(k_l-1)}}(x^{(k_l)})}=0.\]
   Together with the inclusion
  \[0\in \gamma^{(k_l-1)}u_{\gamma^{(k_l-1)}}(x^{(k_l)})+f(x^{(k_l)})+ \partial q(x^{(k_l)}+u_{\gamma^{(k_l-1)}}(x^{(k_l)})),\]
  the continuity of $f$ and the closedness of $\gph\partial q$, it follows that $0\in f(\xb)+\partial q(\xb)$ holds for every accumulation point $\xb$ of the subsequence $x^{(k_l)}$. This shows our second assertion.

  Finally, assume that $\xb$ is an accumulation point of the sequence $x^{(k)}$ such that $0\in H(\xb)$ and $H$ is both metrically regular and \ssstar at $(\xb,0)$ and assume that \eqref{EqBndGamma} holds. Fixing $\kappa'>{\rm reg\;}H(\xb,0)$, by Lemma \ref{LemBndInvMetrReg} we can find  a positive radius $\rho'>0$ such that for all $(x^{(k)},d^{(k)})\in\B_{\rho'}(\xb)\times \R^n$ the Newton direction $\Delta x^{(k)}$ exists and the matrices $A,B$ given by \eqref{EQAB_Appr1} fulfill by virtue of \eqref{EqBndAB} the inequality
  \[\norm{A^{-1}}\norm{(A\vdots B)}_F\leq (1+2\kappa'(l+1))(C_1+C_2\norm{\nabla f(\hat x^{(k)})})^2\leq (1+2\kappa'(l+1))(C_1+C_2 l)^2,\]
  where $l$ denotes the Lipschitz constant of $f$ in $\B_{\rho'}(\xb)$. By Corollary \ref{CorIsolSol} the solution $\bar x$ is isolated and we can choose $\rho'$ possibly smaller such that $\dist{x,H^{-1}(0)}=\norm{x-\xb}$ $\forall x\in\B_{\rho'}(\xb)$.

  By Proposition \ref{PropConv} together with the first equation in \eqref{EqAppr1L}, for every $\epsilon>0$ there is some $\delta>0$ such that
  \begin{align}\nonumber\norm{x^{(k)}+\Delta x^{(k)}-\xb}&\leq \norm{\myvec{\hat x^{(k)}+\Delta x^{(k)}-\xb\\\hat d^{(k)}+\Delta d^{(k)}-\xb}}\leq \epsilon \norm{A^{-1}}\norm{(A\vdots B)}_F\norm{(\hat x^{(k)}-\xb,\hat d^{(k)}-\xb,\hat y_1^{(k)},\hat y_2^{(k)})}\\
  &\leq \epsilon(1+2\kappa'(l+1))(C_1+C_2 l)^2\Big(2+(2+\gamma^{(k)})\big(2+ \frac l{\gamma^{(k)}}\big)\Big)\norm{x^{(k)}-\xb}\label{EqAuxSuperlinear}
  \end{align}
  whenever $x^{(k)}\in \B_{\delta}(\xb)$. In particular, we can find some $0<\bar\delta<\rho'/(1+\frac l{\underline{\gamma}})$ such that
  \[\norm{x^{(k)}+\Delta x^{(k)}-\xb}\leq \frac{1-\nu}{(2+\frac l{\underline{\gamma}})\big(1+\kappa'(\bar\gamma+l)\big)}\norm{x^{(k)}-\xb}< \frac12\norm{x^{(k)}-\xb}\]
  whenever $x^{(k)}\in  \B_{\bar\delta}(\xb)$. We now claim that for every iterate $x^{(k)}\in \B_{\bar\delta}(\xb)$ the Newton step  with step size $\alpha^{(k)}=1$ is accepted.
  Indeed, consider $x^{(k)}\in \B_{\bar\delta}(\xb)$. Then, by \eqref{EqAppr1Lip1} we obtain
  \[\norm{x^{(k)}+u^{(k)}-\xb}=\norm{\hat d^{(k)}-\xb}\leq (1 + \frac l{\underline{\gamma}})\norm{x^{(k)}-\xb}<\rho'\]
  and by the definition of $u^{(k)}$ we have
  \[-\gamma^{(k)}u^{(k)}+f(x^{(k)}+u^{(k)})-f(x^{(k)})\in H(x^{(k)}+u^{(k)}).\]
  Due to metric regularity we conclude
  \begin{align*}\dist{x^{(k)}+u^{(k)},H^{-1}(0)}&=\norm{x^{(k)}+u^{(k)}-\xb}\leq\kappa'\dist{0,H(x^{(k)}+u^{(k)})}\\&\leq \kappa'\norm{-\gamma^{(k)}u^{(k)}+f(x^{(k)}+u^{(k)})-f(x^{(k)})}
  \leq  \kappa'(\bar \gamma+l)\norm{u^{(k)}}\end{align*}
  implying
  \[\norm{x^{(k)}-\xb}\leq \big(1+ \kappa'(\bar \gamma+l)\big)\norm{u^{(k)}}.\]
  Since $u_{\gamma^{(k)}}(\bar x)=0$, we obtain from \eqref{EqLip2}
  \begin{align*}
  \norm{u_{\gamma^{(k)}}(x^{(k)}+\Delta x^{(k)})}&\leq \big(2+\frac l{\underline{\gamma}}\big)\norm{x^{(k)}+\Delta x^{(k)}-\xb}\leq \frac{1-\nu}{\big(1+\kappa'(\bar\gamma+l)\big)}\norm{x^{(k)}-\xb}\leq (1-\nu)\norm{u^{(k)}}\\
  &=(1-\nu)\norm{u_{\gamma^{(k)}}(x^{(k)})}
  \end{align*}
  showing
  \[ r_{\gamma^{(k)}}(x^{(k)}+\Delta x^{(k)})=\sqrt{1+{\gamma^{(k)}}^2}\norm{u_{\gamma^{(k)}}(x^{(k)}+\Delta x^{(k)})}\leq (1-\nu) \sqrt{1+{\gamma^{(k)}}^2}\norm{u_{\gamma^{(k)}}(x^{(k)})}=(1-\nu)r_{\gamma^{(k)}}(x^{(k)}).\]
  From this we conclude that the step size $\alpha^{(k)}=1$ is accepted and thus our claim holds true.
  Now let $\bar k$ denote the first index such that $x^{(\bar k)}$ enters the ball $\B_{\bar\delta}$. Then for all $k\geq \bar k$ we have
  \[x^{(k+1)}=x^{(k)}+\Delta x^{(k)}, \quad \norm{x^{(k+1)}-\xb}\leq \frac 12\norm{x^{(k)}-\xb}\]
  establishing convergence of the sequence $x^{(k)}$ to $\xb$. The superlinear speed of convergence is a consequence of \eqref{EqAuxSuperlinear}.
\end{proof}

In the following subsections we discuss some implementation details and alternatives for the three different mappings $\T$ introduced at the beginning of this section.

\subsubsection{Douglas-Rachford splitting method}

In order that $\T^{\rm DR}_\lambda$ meets the assumptions of Theorem \ref{ThGlobConv} it is sufficient that $f$ is monotone, see, e.g., \cite[Corollary 2]{LiMe79}.
We now discuss a variant of Algorithm \ref{AlgSSNewt_Hybrid} which seems to be slightly more efficient.
Consider the sequences $x^{(k)}$ and $v^{(k)}$ generated by
\[x^{(k+1)}=\T^{\rm DR}_\lambda(x^{(k)}),\ v^{(k)}=(I+\lambda f)(x^{(k)}).\]
Then it is well known, see, e.g., \cite{LiMe79}, that $v^{(k+1)}= \G_\lambda(v^{(k)})$, where
\[\G_\lambda(v): = \Big(J_{\partial q}^\lambda\big(2J_f^\lambda-I\big)+I-J^\lambda_f\Big)(v)\]
with resolvents $J_{\partial q}^\lambda:=(I+\lambda\partial q)^{-1}$, $J_f^\lambda:=(I+\lambda f)^{-1}$.
From \eqref{EqDR} we obtain
\begin{equation}\label{EqAux_v}v^{(k+1)}=\T^{\rm FB}_\lambda(x^{(k)})+ \lambda f(x^{(k)})=x^{(k)}+u_{\frac1\lambda}(x^{(k)})+\lambda f(x^{(k)})=v^{(k)}+u_{\frac1\lambda}(x^{(k)})\end{equation}
implying $u_{\frac1\lambda}(x^{(k)})=v^{(k+1)}-v^{(k)}$.
Thus, $u_{\frac1\lambda}(x^{(k+1)})=v^{(k+2)}-v^{(k+1)}=\G_\lambda(v^{(k+1)})-\G_\lambda(v^{(k)})$ and  by \cite[relation (17)]{LiMe79} we obtain
\begin{align}\nonumber\norm{u_{\frac1\lambda}(x^{(k+1)})}^2&=\norm{\G_\lambda(v^{(k+1)})-\G_\lambda(v^{(k)})}^2\\
\nonumber&\leq \skalp{\G_\lambda(v^{(k+1)})-\G_\lambda(v^{(k)}), v^{(k+1)}-v^{(k)}}\\
\nonumber&\qquad\qquad-\skalp{(I-J_f^\lambda)(v^{(k+1)})- (I-J_f^\lambda)(v^{(k)}),J_f^\lambda(v^{(k+1)})-J_f^\lambda(v^{(k)})}\\
\label{EqDecr_u}&=\skalp{u_{\frac 1\lambda}(x^{(k+1)}),u_{\frac 1\lambda}(x^{(k)})}-\lambda\skalp{f(x^{(k+1)})-f(x^{(k)}),x^{(k+1)}-x^{(k)}}.
\end{align}
For the following analysis we require that $f$ is even strongly monotone, i.e.,
\begin{equation}\label{EqStrongMon}\mu_f:=\inf_{x_1\not=x_2}\frac{\skalp{f(x_1)-f(x_2), x_1-x_2}}{\norm{x_1-x_2}^2}>0.\end{equation}
Recall that $\mu_f=\inf_{x\in\R^n}\mu_f(x)$ with $\mu_f(x)$ given by \eqref{EqMu_f}. Then we obtain from \eqref{EqDecr_u} that
\[\norm{u_{\frac1\lambda}(x^{(k+1)})}^2\leq \norm{u_{\frac1\lambda}(x^{(k+1)})} \norm{u_{\frac1\lambda}(x^{(k+1)})}-\mu_f\norm{x^{(k+1)}-x^{(k)}}^2\]
 and thus
$\norm{u_{\frac1\lambda}(x^{(k+1)})}<\norm{u_{\frac1\lambda}(x^{(k)})}$, i.e., the residual $r_{\frac 1\lambda}(x^{(k)})$ is strictly decreasing. The basic idea is now, to perform alternately a step of the Douglas-Rachford splitting method with parameter $\lambda=\frac 1\gamma$ and then a \ssstar Newton step with line search with parameter $\gamma$, where in the line search we possibly sacrifice a part of the reduction in the residual gained in the Douglas-Rachford splitting step. This procedure is mainly motivated by the excellent performance  of the non-monotone line search heuristic of Algorithm \ref{AlgSSNewtHeur}, where we now have a possibility to control the increment in the residual in order to guarantee convergence.

 \begin{algorithm}[Globally convergent hybrid \ssstar Newton - Douglas Rachford method for VI of the second kind]\label{AlgSSNewt_DR}\mbox{ }\\
Input: A starting point $x^{(0)}$, a parameter $\gamma>0$, line search parameters $0<\nu<1$, $0<\xi<1$, a sequence $\beta_j\downarrow 0$ with $\beta_0=1$  and a stopping tolerance $\epsilon_{tol}>0$.\\
1. Compute $u^{(0)}:=u_{\gamma}(x^{(0)})$ and set the counter $k:=0$. \\
2. If $r_{\gamma}(x^{(2k)})\leq\epsilon_{tol}$, stop the algorithm.\\
3. Compute $x^{2k+1}=\T^{\rm DR}_{1/\gamma}(x^{(2k)})$ and $u^{(2k+1)}:=u_{\gamma}(x^{(2k+1)})$.\\
4. \parbox[t]{\myAlgBox}{Compute $G^{(2k+1)}$ fulfilling \eqref{Eq_G_k} and the Newton direction $\Delta x^{(2k+1)}$ by solving \eqref{EqNewtonStepAppr1}. Determine the step size $\alpha^{(2k+1)}$ as the first element from the sequence $\beta_j$ satisfying
\[\norm{u_{\gamma}(x^{(2k+1)}+\beta_j\Delta x^{(2k+1)})}\leq (1-\nu \beta_j)\big(\xi\norm{u^{(2k)}}+(1-\xi)\norm{u^{(2k+1)}}\big).\]}
5. Set $x^{(2k+2)}=x^{(2k+1)}+ \alpha^{(2k+1)} \Delta x^{(2k+1)}$ and $u^{(2k+2)}:=u_{\gamma}(x^{(2k+2)})$.\\
6. Increase the  counter $k:=k+1$ and go to Step 2.
\end{algorithm}

\begin{theorem}\label{ThGlobDR}
  If $f$ is strongly monotone then Algorithm \ref{AlgSSNewt_DR} is well defined. If, in addition, $f$ is Lipschitzian on the set $S^{(0)}:=\{x\mv\norm{u_\gamma(x)}\leq \norm{u_\gamma(x^{(0)})}\}$, then  the sequence $u^{(2k)}$ converges at least Q-linearly to $0$ and the sequence $x^{(j)}$ converges at least R-linearly to the unique solution $\xb$ of \eqref{EqVI2ndK}. If $H$ is \ssstar at $(\xb,0)$, then convergence of the sequence $x^{(2k)}$ is Q-superlinear.
\end{theorem}
\begin{proof}
Since $f$ is strongly monotone, for every $k$ the Newton direction $\Delta x^{(2k+1)}$ is well defined by Corollary \ref{CorExistNewtonDir}. Further, from \eqref{EqDecr_u} it follows that $\norm{u^{(2k+1}}<\norm{u^{(2k)}}$ implying $\norm{u^{(2k+1)}}< \xi\norm{u^{(2k)}}+(1-\xi)\norm{u^{(2k+1)}}<\norm{u^{(2k}}$. Since the function $\phi^{(2k+1)}(\alpha):=\norm{u_{\gamma}(x^{(2k+1)}+\alpha\Delta x^{(2k+1)})}$ is continuous by virtue of \eqref{EqLip2} and $\phi^{(2k+1)}(0)=\norm{u^{(2k+1)}}$, we conclude that $\phi^{(2k+1)}(\alpha)<(1-\nu\alpha)(\xi\norm{u^{(2k)}}+(1-\xi)\norm{u^{(2k+1)}})$ for all sufficiently small $\alpha>0$. Thus, also the step size $\alpha^{(2k+1)}$ is well defined and hence so is the whole algorithm. Note that
\begin{equation}\label{EqAuxU}\norm{u^{(2k+2)}}=\phi^{(2k+1)}(\alpha^{(2k+1)})<\xi\norm{u^{(2k)}}+(1-\xi)\norm{u^{(2k+1)}}<\norm{u^{(2k)}}\end{equation}
and therefore $x^{(j)}\in S^{(0)}$ $\forall j>0$. Denoting by $l$ the Lipschitz constant of $f$ on $S^{(0)}$, we conclude that $\norm{\nabla f(x^{(2k+1)})}\leq l$ $\forall k$ because of $x^{(2k+1)}\in\inn S^{(0)}$.
From \eqref{EqAux_v} we deduce
\[(I+\frac 1\gamma f)(x^{(2k+1)})-(I+\frac 1\gamma f)(x^{(2k)})=u^{(2k)}\]
implying
\[\norm{u^{(2k)}}>\norm{x^{(2k+1)}-x^{(2k)}}\geq \frac \gamma{\gamma+l}\norm{u^{(2k)}}.\]
 Using \eqref{EqDecr_u} we obtain
\[\norm{u^{(2k+1)}}^2\leq \skalp{u^{(2k+1)},u^{(2k)}}-\frac{\mu_f}{\gamma}\norm{x^{(2k+1)}-x^{(2k)}}^2\leq \norm{u^{(2k+1)}}\norm{u^{(2k)}}- \frac{\mu_f\gamma}{(\gamma+l)^2}\norm{u^{(2k)}}^2\]
and consequently
\[\norm{u^{(2k+1)}} < \tau \norm{u^{(2k)}}\quad\mbox{with}\quad\tau:=1-\frac{\mu_f\gamma}{(\gamma+l)^2}<1.\]
Combining this estimate with \eqref{EqAuxU} we obtain
\[\norm{u^{(2k+2)}}<\bar\tau\norm{u^{(2k)}}\quad\mbox{with}\quad \bar\tau:=\xi+\tau(1-\xi))<1\]
and Q-linear convergence of the sequence $u^{(2k)}$ is established. By Corollary \ref{CorExistNewtonDir} we obtain
\[\norm{\Delta x^{(2k+1)}}\leq \Big(1+\frac 1{\mu_f}\norm{\nabla f(x^{(2k+1)})-I}\Big)\max\{1,\gamma\}\norm{u^{(2k+1)}}\leq \Big(1+\frac{l+1}{\mu_f}\Big)\max\{1,\gamma\}\norm{u^{(2k+1)}}\]
and together with $\alpha^{(2k+1)}\leq1$ we have
\begin{align*}\norm{x^{(2k+2)}-x^{(2k)}}&\leq \norm{x^{(2k+2)}-x^{(2k+1)}}+\norm{x^{(2k+1)}-x^{(2k)}}\leq \norm{\Delta x^{(2k+1)}}+\norm{u^{(2k)}}\leq C \norm{u^{(2k)}}\end{align*}
with $C:=\tau \big(1+\frac {l+1}{\mu_f}\big)\max\{1,\gamma\}+1$. This implies
\[\norm{x^{(2j)}-x^{(2k)}}\leq C\sum_{i=k}^{j-1}\norm{u^{(2i)}}\leq \frac C{1-\bar\tau}\norm{u^{(2k)}}\]
for all $0<k<j$. Thus $x^{(2k)}$ is a Cauchy sequence and therefore convergent to some $\tilde x$. By continuity of $u_\gamma(\cdot)$ we have
$u_\gamma(\tilde x)=\lim_{k\to\infty} u_\gamma(x^{(2k)})=\lim_{k\to\infty}u^{(2k)}=0$ and hence $0\in H(\tilde x)$. But by strong monotonicity of $H$ the solution of \eqref{EqVI2ndK} is unique and $\tilde x=\bar x$ follows. Further we have
\[\norm{x^{(2k)}-\xb}\leq \frac C{1-\bar\tau}\bar\tau^k\norm{u^{(0)}},\]
\[\norm{x^{(2k+1)}-\xb}\leq \norm{x^{(2k)}-\xb}+\norm{x^{(2k+1)}-x^{(2k)}}\leq \norm{x^{(2k)}-\xb}+\norm{u^{(2k)}}\leq\Big(1+\frac C{1-\bar\tau}\Big)\bar\tau^k\norm{u^{(0)}}\]
and R-linear convergence of $x^{(j)}$ to $\bar x$ with convergence factor $\sqrt{\bar\tau}$ follows.

There remains to show the superlinear convergence of the sequence $x^{(2k)}$. Strong monotonicity of $f$ implies that $H$ is a maximal strongly monotone mapping and hence it is (strongly) metrically regular around $(\xb,0)$. Using similar argument as in the proof of Theorem \ref{ThGlobConv} we can conclude that $\alpha^{(k)}=1$ for all $k$ sufficiently large and
$\lim_{k\to\infty}\norm{x^{(2k+2)}-\xb}/\norm{x^{(2k+1)}-\xb}=0$. Further,
\[\frac{\norm{x^{(2k+1)}-\xb}}{\norm{x^{(2k)}-\xb}}\leq 1+\frac{\norm{u^{(2k)}}}{\norm{x^{(2k)}-\xb}}\leq 3+\frac l\gamma\]
by \eqref{EqAppr1Lip2} and $\lim_{k\to\infty}\norm{x^{(2k+2)}-\xb}/\norm{x^{(2k)}-\xb}=0$ follows. This completes the proof.
\end{proof}

\begin{remark}
  The factor $\tau$ appearing in the proof of Theorem \ref{ThGlobDR} is the smallest when $\gamma=l$, the Lipschitz constant of $f$. This is in accordance with our practical experience with the heuristic Algorithm \ref{AlgSSNewtHeur} that a choice $\gamma^{(k)}\approx\norm{\nabla f(x^{(k)})}$ yields good results.
\end{remark}

\begin{remark}
  The requirement that $f$ is Lipschitzian on $S^{(0)}$ is, e.g., fulfilled if $S^{(0)}$ is bounded. In particular, since $x+u_\gamma(x)\in\dom\partial q$, this is the case when $\dom\partial q$ is bounded.
\end{remark}

\subsubsection{Forward-backward splitting method}

Most of the research on  forward-backward splitting methods has relied on
assumptions of strong monotonicity, cf. \cite{Gab83}. E.g., when $f$ is Lipschitzian on $\dom\partial q$, and either $f$ or $\partial q$ is strongly monotone, then $\T_\lambda^{\rm FB}$ fulfills the requirements of Theorem \ref{ThGlobConv} provided $\lambda$ is chosen sufficiently small, see, e.g., \cite{ChRo97}. Using \cite[Theorem 2.4]{ChRo97}, it is not difficult to show, that for the sequence $x^{(k+1)}=\T^{\rm FB}_\lambda(x^{(k)})$ we have
\[\norm{u_{1/\lambda}(x^{(k+1)})}<\tau \norm{u_{1/\lambda}(x^{(k)})}\]
with some factor $\tau<1$ for $\lambda$ small enough. We could proceed in the same way as we have used for the Douglas-Rachford splitting method, but we omit to do this for the following reason. When we are forced to choose $\lambda$ very small, in particular when $\frac 1\lambda$ is much larger than the Lipschitz constant of $f$, our numerical experiments do not show a favourable behaviour compared with our \ssstar Newton approaches based on $\T^{\rm DR}_\lambda$ and $T^{\rm PM}_\lambda$. On the other hand, if we are allowed to choose $\lambda$ comparatively large, then the pure forward-backward method shows a good convergence behaviour and we need not to use the \ssstar Newton method at all.

\subsubsection{Hybrid projection-proximal point algorithm}

When using $\T^{\rm PM}_\gamma$, we only need monotonicity of $H$, monotonicity of $f$ is not required.

Consider a sequence $x^{(k+1)}=\T^{\rm PM}_{\gamma^{(k)}}(x^{(k)})$. It follows from \cite[Theorem 2.2]{SolSv99} that the following two conditions are sufficient in order to meet the assumptions of Theorem \ref{ThGlobConv}:
\begin{enumerate}
  \item $\gamma^{(k)}>0$ $\forall k$ and $\sum_{k=0}^{\infty} (\gamma^{(k)})^{-2}=\infty$.
  \item There is some $\sigma\in[0,1)$ such that
  \begin{align}\label{EqSolSvGamma}\norm{f(x^{(k)}+u^{(k)})- f(x^{(k)})}\leq \sigma\max\{\norm{-\gamma^{(k)}u^{(k)}+ f(x^{(k)}+u^{(k)})- f(x^{(k)})},
  \gamma^{(k)}\norm{u^{(k)}}\}\quad \forall k,\end{align}
  where $u^{(k)}:=u_{\gamma^{(k)}}(x^{(k)}))$.
\end{enumerate}

We now demonstrate that both conditions can be fulfilled by setting $\gamma^{(k)}=\hat\gamma$ $\forall k$ for any $\hat\gamma\geq\hat l/\sigma$ with
\[\hat l:=\max\{\norm{\nabla f(x)}\mv x\in \B_{2\norm{x^{(0)}-\xb}}(\xb)\},\]
where $\xb$ denotes any solution of \eqref{EqVI2ndK} and $\sigma\in(0,1)$ is arbitrarily fixed. It follows that $f$ is Lipschitzian on $\B_{2\norm{x^{(0)}-\xb}}(\xb)$ with constant $\hat l$. Of course, condition 1. is trivially fulfilled and there remains to show the second one. Consider any iterate $x^{(k)}\in \B_{\norm{x^{(0)}-\xb}}(\xb)$.  By \eqref{EqLip1} we obtain
\[\norm{x^{(k)}+u^{(k)}-\xb}\leq \norm{x^{(k)}-\xb} +\frac{\norm{f(x^{(k)})-f(\xb)}}{\hat\gamma}\leq (1+\sigma)\norm{x^{(k)}-\xb}\]
and therefore
\[\norm{f(x^{(k)}+u^{(k)})- f(x^{(k)})}\leq \hat l \norm{u^{(k)}}\leq \sigma\hat\gamma\norm{u^{(k)}}.\]
By \cite[Lemma 2.1]{SolSv99} we have $\norm{x^{(k+1)}-\xb}\leq \norm{x^{(k)}-\xb}$ and our claim follows by induction.

In practice we choose $\gamma^{(k)}$ not constant in every iteration, but we try it to adjust it to a local Lipschitz constant of $f$ near $x^{(k)}$. E.g., we can choose $\gamma^{(k)}$ as the first element of a sequence $\chi_j\uparrow\infty$ such that $\chi_j\geq \norm{\nabla f(x^{(k)})}$ and inequality \eqref{EqSolSvGamma} holds.

\section{Numerical experiments}
All variants of the semismooth* Newton method presented in the preceding sections have been extensively tested by means of a wide range of examples. In this section we will show first the behavior of Algorithm 1 via a low-dimensional example with an economic background. Thereafter we will illustrate the efficiency of the family of methods, presented in Section 5, by means of an artificially constructed set of problems having a variable scale.

\subsection{An economic equilibrium}
In \cite{OV} the authors considered an evolution process in an oligopolistic market, where the players (firms) adapt their strategies (productions) according to changing external parameters (input prices etc.). In their decisions, however, they must take into account that each change of production may be associated with some costs, see \cite{Fl}. As derived in \cite{OV}, the respective Cournot-Nash equilibrium at some time instant is governed by GE \eqref{EqVI2ndK} with
$$q = \tilde{q} + \delta_A, \qquad
A = \prod\limits_{i=1}^n A_i, \qquad \tilde{q}(x) = \sum\limits_{i=1}^n \tilde{q_i}(x_i)$$
where $n$ denotes the number of players. The strategy sets $A_i$ are nonempty and compact intervals $[b_i, d_i]$ and $\tilde{q_i}(x_i) = \beta_i | x_i - a_i |$ for some non-negative reals $\beta_i$ and parameters $a_i\in A_i, i=1,2,\ldots,n$. Mapping $f$ is continuously differentiable on an open set containing $\dom \partial{q} = A$ and its description can be found in, e.g., \cite{MSS} and \cite{OKZ}; see also \cite{OV}, where the values of $a=(a_1,a_2,\ldots,a_n)$ and $\beta=(\beta_1,\beta_2,\ldots,\beta_n)$ are specified. The implementation of the semismooth* Newton method described in Section 4 has been first applied to the problem formulation from \cite{OV}, where all production cost functions are convex and f is strongly monotone on $A$. Thereafter we have replaced the production cost function of the first player by a (more realistic) concave one. As a consequence, mapping $f$ has lost its monotonicity on $A$ and the respective GE \eqref{EqVI2ndK} might have possibly multiple solutions with not all of them being necessarily Cournot-Nash equilibria. It has turned out, however, that, in our example, the Jacobian of $f$ is positive definite at the obtained solution. This implies in particular that this point is a Cournot-Nash equilibrium and the respective multifunction $H$ is metrically regular there.

It is easy to see that in both cases the resulting multifunction $\F$ is semismooth* at any point of its graph. Since $q$ is a separable function, one has that
$\partial{q(x)}=\prod\limits_{i=1}^n \partial{q_i(x_i)}$ and the sets $\gph\partial{q_i}$ attain the form depicted in Figure \ref{fig:gph_partial}.


Next we provide a simple formula for the computation of the matrix $G$ needed in the Newton step. This formula, however, will be given for a more general situation considered in the family of test examples discussed in Subsection \ref{subsection_6.2}. Assume that each $\gph \partial q_i$ is a polygonal line in $\mathbb{R}^2$ connecting the given points
\begin{equation}\label{EqPoly} (\xi^i_1, -\infty), (\xi^i_1, \eta^i_1), (\xi^i_2, \eta^i_2),
\dots,
(\xi^i_{2 m_i-1}, \eta^i_{2 m_i-1}),(\xi^i_{2 m_i}, \eta^i_{2 m_i}), (\xi^i_{2 m_i}, \infty), \quad i=1, 2, \ldots, n,
\end{equation}
for some integer $m_i \geq 1$. Further suppose that
\begin{eqnarray*}
&&\Delta \xi^i_j :=  \xi^i_{j+1} - \xi^i_j
\begin{cases}
>0 \quad \mbox{if }j \mbox{ is odd} \\
=0 \quad \mbox{if }j \mbox{ is even}, \\
\end{cases} \\
&&\Delta \eta^i_j :=  \eta^i_{j+1} - \eta^i_j  \begin{cases}
\geq 0 \quad \mbox{if }j \mbox{ is odd} \\
>0 \quad \mbox{if }j \mbox{ is even}. \\
\end{cases}
\end{eqnarray*}
Observe that $\Delta \xi^i_j + \Delta \eta^i_j>0$ holds for all $i$ and all $j=1,\ldots,2m_i-1$.
It follows that a polygonal line given this way is monotone increasing and, consequently, $q_i$ is a convex piecewise linear-quadratic function with $\dom q_i = [ \xi_1^i, \xi^i_{2 m_i}    ]$. Clearly, in the example depicted in Figure \ref{fig:gph_partial} one has
$$m_i=2, \xi^i_1=b_i, \xi^i_2=\xi^i_3=a_i, \xi^i_4=d_i, \eta_1^i=\eta_2^i=-\beta_i, \eta_3^i=\eta_4^i=\beta_i.$$
\begin{figure}
\centering
\begin{tikzpicture}[scale=5]
  \draw[-] (-0.25,0) -- (2.25,0) coordinate (x axis);
  \draw[-] (0,-0.75) -- (0,0.75) coordinate (y axis);

  \draw[red,thick, -] (0.25,-0.75) |- (0.9,-0.5) |- (0.9,0.5) |- (2,0.5) -| (2,0.75) ;
  \draw[dotted] (0,0.5) -- (0.9,0.5) ;
  \draw[dotted] (0,-0.5) -- (0.25,-0.5) ;
  \draw[dotted] (0.25,0) -- (0.25,-0.5) ;
  \draw[dotted] (2,0) -- (2, 0.5) ;

  \draw (0,0) node[below left] {$0$} ;
  \draw (0,0.5) node[left] {$\beta_i$} ;
  \draw (0,-0.5) node[left] {$-\beta_i$} ;
  \draw (0.25,0) node[below left] {$b_i$} ;
  \draw (0.9,0) node[below left] {$a_i$} ;
  \draw (2,0) node[below ] {$d_i$} ;

  \end{tikzpicture}
  \caption{$\gph \partial q_i$ with $A_i=[b_i, d_i]$. }
  \label{fig:gph_partial}
\end{figure}
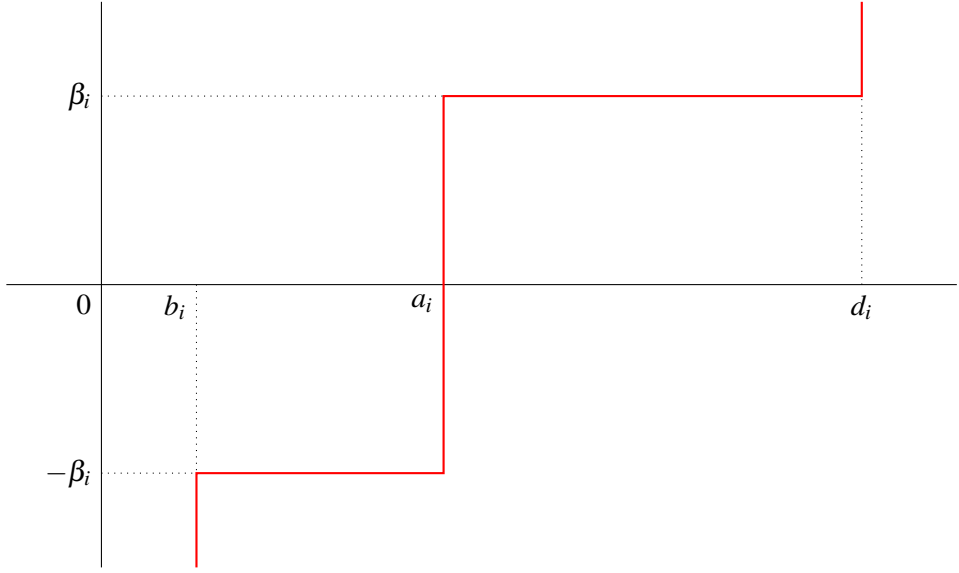
With this problem structure it is not difficult to compute the quantity $u_\gamma(x)$ as follows. For each $i=1,\ldots,n$ we denote by $u_i$ and $f_i$ the  $i$th component of $u_\gamma(x)$ and $f(x)$, respectively, i.e., $u_i$ solves the inclusion
\[0\in \gamma u_i +f_i +\partial q_i(x_i+u_i)=\gamma(x_i+u_i)+(f_i-\gamma x_i)+\partial q_i(x_i+u_i)\]
Let
\[j_i=\min\{j\in\{1,\ldots,2m_i\}\mv \gamma x_i-f_i < \gamma \xi_j^i+\eta_j^i\}\quad (\mbox{$\infty$, if $\gamma x_i-f_i \geq \gamma \xi_{2m_i}^ i+\eta_{2m_i}^i$}).\]
Then
\[u_i+x_i=\begin{cases}\xi_1^i&\mbox{if $j_i=1$,}\\
\xi_{j_i-1}^i+t_i\Delta\xi_{j_i-1}^i&\mbox{if $1<j_i\leq 2m_i$,}\\
\xi_{2m_i}^i&\mbox{if $j_i=\infty$,}\end{cases}
\ \mbox{with}\ t_i=\frac{\gamma x_i-f_i-(\gamma\xi_{j_i-1}^i+\eta_{j_i-1}^i)}{\gamma \Delta\xi_{j_i-1}^i+\Delta\eta_{j_i-1}^i}.\]

The matrix $G$, needed in the Newton step, can be computed as follows.
\begin{proposition}\label{coder}
Let $(x,x^*)\in \gph \partial{q}$ and $G$ be $n \times n$ diagonal matrix with entries
\begin{equation}
G_{ii}=\begin{cases}
1   & \mbox{ if  $x_i \in \mathcal{A}_i:= \{ \xi^i_1  \} \cup \{ \xi^i_{2m_i} \} \cup \{ \xi^i_j \, | \, \Delta \xi^i_j = 0   \}$ } \\
\frac{\Delta \eta^i_j}{\Delta \xi^i_j + \Delta \eta^i_j} &\mbox{ if  $x_i \in [\xi^i_j, \xi^i_{j+1}] \setminus \mathcal{A}_i$ and  $j \in \{ 1, \ldots, 2 m_i \}$  is odd} \end{cases}\label{G_i_2}
\end{equation}
for $i=1, 2, \dots, n$.
Then $G$ fulfills the conditions stated in Theorem 3.6.
\end{proposition}
\begin{proof}
Observe first that the set
$$ \{ (x_i, x_i^*) \in \gph \partial q_i \, | \, x_i \in \mathcal{A}_i \}$$
comprises all points of $\gph \partial q_i$ lying in its vectical line segments and
$$\bigcup\limits_{\substack{j=1 \\ j {\mbox{ \scriptsize odd}} }}^{2 m_i} [\xi^i_j, \xi^i_{j+1}] \setminus \mathcal{A}_i = \{ x_i \in \dom \partial q_i \, | \, \partial q (x_i) \mbox{ is a singleton} \}.$$
Next let us notice that
$\gph  D^*(\partial{q})(x,x^*) =
\prod\limits_{i=1}^n
\gph  D^*(\partial{q_i})(x_i,x^*_i)$, where
$$ \gph D^*(\partial q_i)(x_i,x^*_i)=\{(u,v) | (v,-u)\in N_{\gph\partial{q_i}}(x_i,x^*_i)\} $$ and
$$N_{\gph\partial{q_i}}(x_i,x^*_i) \supset
\begin{cases}
\{0\} \times \mathbb{R} & \mbox{provided } x_i \in \mathcal{A}_i \\
\mathbb{R} (\Delta \eta^i_j, -\Delta \xi^i_j) & \mbox{ provided } x_i \in [\xi^i_j, \xi^i_{j+1}] \setminus \mathcal{A}_i \mbox{ and } j \in \{ 1, \ldots, 2 m_i -1\}  \mbox{ is odd}.
\end{cases}$$
From this analysis it follows that, in order to fulfill inclusion \eqref{EqRgeG}, it suffices to construct $G$ as a diagonal matrix, where $G_{ii}=1$ provided $x_i \in \mathcal{A}_i$. Otherwise, if $x_i \in [\xi^i_j, \xi^i_{j+1}] \setminus \mathcal{A}_i$ for some odd $j \in \{ 1, \ldots, 2 m_i -1\}$, then we put $G_{ii}$ as the (unique) solutions of the equation
$$ \frac{\Delta \eta^i_j}{\Delta \xi^i_j} = \frac{G_{ii}}{1-G_{ii}}. $$
The above equation is well-posed because $\Delta \xi^i_j > 0$ for $j$ odd and leads to the second line in formula \eqref{G_i_2}. By construction, all elements $G_{ii}$ belong to $[0, 1]$, which implies that G is positive semidefinite and $\norm{G} \leq1$. Thus, the proof is complete.
\end{proof}
In the example depicted in Figure \ref{fig:gph_partial} one obtains in this way that
\begin{equation}\label{G_i_simplified}
G_{ii}=
\begin{cases}
1   &  \mbox{if } x_i \in \mathcal{A}_i:= \{ b_i \} \cup \{ a_i \} \cup \{ d_i \}  \\
0   &  \mbox{if } x_i \in (b_i, a_i) \cup (a_i, d_i),
\end{cases}
\end{equation}
which has been used in the computations discussed below.

\if{
\begin{proposition}\label{coder}
Let $(x,x^*)\in \gph \partial{q}$ and for $i=1,2,\ldots,n$
\begin{equation}\label{G_i}
G_i=
\begin{cases}
e_i   &  \mbox{if } (x_i,x^*_i) \in( \{b_i \}\times(-\infty,-\beta_i])\cup(\{a_i\}\times[-\beta_i,\beta_i])\cup(\{d_i\}\times[\beta_i,\infty)),\\
0 & \mbox{otherwise}.
\end{cases}
\end{equation}
Then $G$ fulfills the conditions stated in Theorem 3.6.
\end{proposition}
\begin{proof}
Observe first that
$\gph  D^*(\partial{q})(x,x^*) =
\varprod \limits_{i=1}^n
\gph  D^*(\partial{q_i})(x_i,x^*_i)$, where
$$ \gph D^*(\partial q_i)(x_i,x^*_i)=\{(u,v) | (v,-u)\in N_{\gph\partial{q_i}}(x_i,x^*_i)\} $$ and
$$N_{\gph\partial{q_i}}(x_i,x^*_i) \supset
\begin{cases}
\mathbb{R} \times \{0\}   &  \mbox{at the vertical line segments} \\
\{0\} \times \mathbb{R} & \mbox{at the horizontal line segments}
\end{cases}$$
of $\gph \partial{q}, i=1, \ldots, n$.
From this it immediately follows that $G$, composed according to \eqref{G_i}, fulfills the inclusion
$$ \rge{G} \subset \gph D^*(\partial{q})(x,x^*).$$
Since $G$ is diagonal with non-negative eigenvalues less or equal 1, all conditions of Theorem 3.6 are satisfied and we are done.
\end{proof}

}\fi

\newcolumntype{R}{>{\raggedleft\arraybackslash}X}
\newcolumntype{G}{>{\raggedleft\arraybackslash}X}
\begin{table}[ht]
\normalsize
\begin{tabularx}{\linewidth}{| X | X | R R R R R | }
\hline
& $\;i$ & 1 & 2 & 3 & 4 & 5 \\
\hline
convex $c_1$ & strategies
& 49.411& 51.140& 54.236& 48.054& 43.095\\
from \cite{OV}
& objectives
& -377.239& -459.943& -639.952& -503.445& -507.100\\
 & value $q_i(x_i)$
& 0.800& 0& 5.831& 0& 0\\
\hline
concave $c_1$ & strategies
& 97.191& 51.140& 51.320& 43.254& 39.470\\
from \eqref{concave_c1} & objectives
& -427.320& -325.770& -491.385& -367.618& -387.836\\
 & value $q_i(x_i)$ & 24.691& 0& 0& 0& 0\\
\hline
\end{tabularx}
\caption{Cournot-Nash equilibrium strategies $x_i$, the corresponding objective values and costs of change $q_i(x_i)$ in case of convex and concave cost function $c_1$.} \label{tab:Cournot}
\end{table}

Next we will present the numerical results for both the monotone and non-monotone case discussed above. In the former one we have used the data from \cite[Section 5.1]{OV} with $t=1$ and started the iteration process at the initial iterate $x^{(0)}=(75, 75, \dots, 75)$. The results are displayed in the upper part of Table 1. Concerning the non-monotone case, we have replaced the original convex production cost function $c_1$ by a concave one, given by
\begin{equation} \label{concave_c1}
c_1(x_1):= -(1/50) x_1^2 + 15 x_1
\end{equation}
and started from the same vector $x^{(0)}$. The results are displayed in the lower part of Table 1. In both cases we have set
$\gamma=1$ in the approximation step and, as the stopping criterion, we have used the condition $|| u || \leq \epsilon = 10^{-10}$, where $u$ is the output of the approximation step.


\begin{figure}[bh]
\begin{center}
\begin{tikzpicture}[scale=0.82]
\begin{semilogyaxis}
[xlabel=$k$,ylabel=$||u_k||$]
\addplot[color=blue,mark=x] coordinates {
(0,  20.853178402314793)
 (1,  5.369773732551855)
 (2,  1.468940319407929)
 (3,  0.755432776659333)
 (4,  0.000016094606048)
 (5,  0.000000000000557)
};
\end{semilogyaxis}
\end{tikzpicture}
\qquad
\begin{tikzpicture}[scale=0.82]
\begin{semilogyaxis}
[xlabel=$k$,ylabel=$||u_k||$,
]
\addplot[color=blue,mark=x] coordinates {
(0,  20.853178402314793)
 (1,  4.798082169927617)
 (2,  3.863410857656254)
 (3,  5.366813340441373)
 (4,  1.819430536754804)
 (5,  0.902708819979616)
 (6,  0.000504712243625)
 (7,  0.000000061147895)
 (8,  0.000000000000001)
};
\end{semilogyaxis}
\end{tikzpicture}
\caption{Convergence of $||u^{(k)}||$ in case of convex (left) or concave (right) cost functions $c_1$. }
    \label{fig:convergence}
    \end{center}
\end{figure}
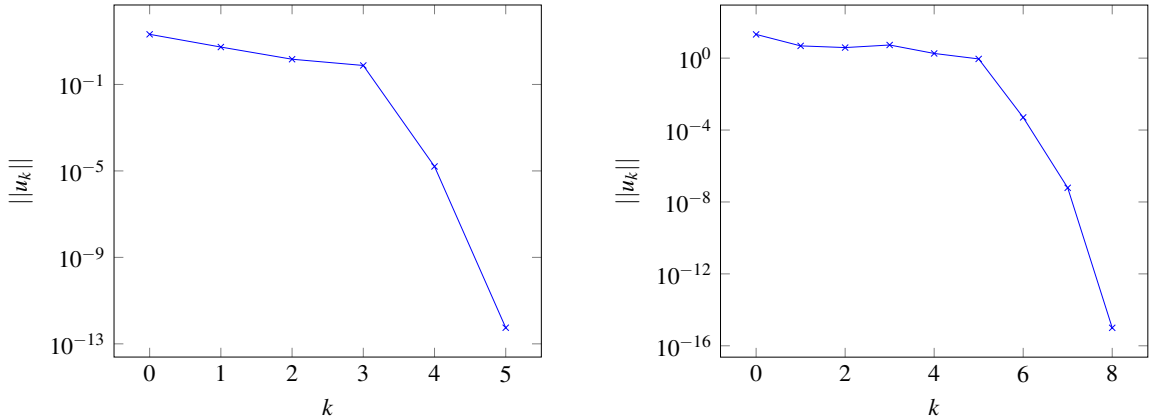

Figure 2 illustrates very well the superlinear convergence of Algorithm 1 whenever one reaches the respective neighborhood of the solution. Note that in our implementation we moved the stopping criterion behind the approximation step in order to dispose with the actual value of $u$. Further observe that the lowest eigenvalue of the symmetrized Jacobian of $f$ at the solution is only 0.033474  in the non-monotone case whereas it amounts to 0.207905 in the case of convex $c_1$.

\if{Numerical results were generated by a Matlab code available for download at:
\begin{center}
{\em \url{https://www.mathworks.com/matlabcentral/fileexchange/72771}} .
\end{center}}\fi

\subsection{Randomly constructed test problems}\label{subsection_6.2}

Given  the problem dimension $n$ and a parameter $\beta>0$, we construct an instance of  the problem \eqref{EqVI2ndK} as follows.
\begin{enumerate}
  \item We  randomly  compute an $n\times n$ matrix $C$ with elements uniformly distributed in $[-1,1]$ and set
  \[f(x) = \nabla h(x) +(C-C^T)x \mbox{ with } h(x)=(x^TAx)^2 ,\ A=\frac\beta {n} CC^T.\]
  Note that $f$ is a maximal monotone operator but the skew-symmetric part $(C-C^T)x$ dominates $f(x)$ for small values of $\beta$.
  \item The set-valued part is of the form $\partial q(x)=\prod_{i=1}^n\partial q_i(x)$, where each $\gph \partial q_i$ is a polygonal line in $\R^2$ connecting the points
  \eqref{EqRes1}
  as described in the previous subsection. The integer $m_i$ is randomly chosen in $[1,10]$, $\xi_1$ is randomly chosen from $[-\frac{m_i}2, \frac{m_i}2]$, $\eta_1$ is randomly chosen from $[-\frac{3\beta m_i}2,0]$
  and
  \[ \Delta\xi_j:=\xi_{j+1}-\xi_j\begin{cases}\in[0,1]&\mbox{if $j$ even}\\=0&\mbox{if $j$ odd}\end{cases},\quad \Delta\eta_j:=\eta_{j+1}-\eta_j\in[0,\beta], \quad j=1,\ldots,2m_i-1.\]
\end{enumerate}

\if{In case when $1<j_i\leq 2m_i$ we have $t_i\in[0,1)$ by construction. If $t_i>0$, then $\widehat N_{\gph \partial q_i}(x_i+u_i, -f_i-\gamma u_i)=\R(\Delta\eta_{j_i-1},-\Delta\xi_{j_i-1})$. On the other hand, when $t_i=0$ there still holds $\R(\Delta\eta_{j_i-1},-\Delta\xi_{j_i-1})\subseteq N_{\gph \partial q_i}(x_i+u_i, -f_i-\gamma u_i)$. In case when $j_i=1$ or $j_i=\infty$ we have $\R(0,1)\subseteq N_{\gph \partial q_i}(x_i+u_i, -f_i-\gamma u_i)$.
As a positive semidefinite matrix $G$ with $\norm{G}\leq 1$ and $\rge{G}\subseteq \gph D^*(\partial q)(x+u_\gamma(x), -f(x)-\gamma u_\gamma(x))$ we can therefore  choose a diagonal matrix
with entries
\[G_{ii}=\begin{cases}1&\mbox{if $j_i=1$ or $j_i=\infty$,}\\
\frac{\Delta\eta_{j_i-1}}{\Delta\xi_{j_i-1}+\Delta\eta_{j_i-1}}&\mbox{else,}\end{cases}\quad i=1,\ldots,n.\]
}\fi

The mapping $\partial q$ is a maximal strongly monotone mapping (with probability 1) and so is the mapping $H$ as well. Thus, the GE \eqref{EqVI2ndK} has always a unique solution $\xb$, but the problem characteristic will change with $\beta$ and $n$.

For each $n\in\{150,600,2400\}$ and each $\beta\in\{1,10^{-2},10^{-4}\}$ we constructed 5 test problems. In  Table \ref{TabCharVal} we list the mean values for some characteristic values for different combinations of $\beta$ and $n$.
\begin{table}[h]
\begin{center}
\begin{tabular}{|c@{}|@{}c@{}|@{}c@{}||@{}c@{}|@{}c@{}||@{}c@{}|@{}c|}
\hline
&\multicolumn{2}{|c||}{$\beta=1$}&\multicolumn{2}{|c||}{$\beta=0.01$}&\multicolumn{2}{|c|}{$\beta=10^{-4}$}\\
\hline
$n$&$\begin{array}{c}\norm{\nabla f(\xb)}\\\norm{\nabla f(\xb)^{-1}}\end{array}$ & $\begin{array}{c}\mu_f(\xb)\\ \mu_q(\xb)\end{array}$
&$\begin{array}{c}\norm{\nabla f(\xb)}\\\norm{\nabla f(\xb)^{-1}}\end{array}$ & $\begin{array}{c}\mu_f(\xb)\\ \mu_q(\xb)\end{array}$
&$\begin{array}{c}\norm{\nabla f(\xb)}\\\norm{\nabla f(\xb)^{-1}}\end{array}$ & $\begin{array}{c}\mu_f(\xb)\\ \mu_q(\xb)\end{array}$  \\
\hline
150&$\begin{array}{c}83.7\\0.86\end{array}$&$\begin{array}{c}5.2\times10^{-4}\\5.8\times10^{-2}\end{array}$
&$\begin{array}{c}19.5\\31.1\end{array}$&$\begin{array}{c}1.1\times 10^{-7}\\5.9\times10^{-4}\end{array}$
&$\begin{array}{c}19.4\\17.7\end{array}$&$\begin{array}{c}4.2\times 10^{-11}\\5.2\times10^{-6}\end{array}$\\
\hline
600&$\begin{array}{c}290\\0.496\end{array}$&$\begin{array}{c}1.0\times10^{-4}\\7.5\times 10^{-3}\end{array}$
&$\begin{array}{c}39.7\\42.9\end{array}$&$\begin{array}{c}9.5\times 10^{-9}\\8.3\times10^{-5}\end{array}$
&$\begin{array}{c}39.1\\25.7\end{array}$&$\begin{array}{c}4.5\times 10^{-12}\\4.4\times10^{-7}\end{array}$\\
\hline
2400&$\begin{array}{c}1202\\0.387\end{array}$&$\begin{array}{c}2.7\times10^{-5}\\2.9\times 10^{-3}\end{array}$
&$\begin{array}{c}79.5\\18.7\end{array}$&$\begin{array}{c}2.8\times 10^{-9}\\8.8\times10^{-5}\end{array}$
&$\begin{array}{c}79.6\\125\end{array}$&$\begin{array}{c}6.4\times 10^{-13}\\5.3\times10^{-7}\end{array}$ \\
\hline
\end{tabular}
\end{center}
\caption{Some characteristic values for the problems} \label{TabCharVal}
\end{table}

Note that $\norm{\nabla f(\xb)}$ acts as a local Lipschitz constant for $f$ near the solution $\xb$, whereas $\mu_f(\xb)$ and $\mu_q(\xb)$ given by \eqref{EqMu_f},\eqref{EqMu_q} are constants for local strong monotonicity for $f$ and $\partial q$.

Theoretically the global convergent methods  obey linear convergence properties. However, the available bounds for the convergence factors depend in some way on the ratio of the Lipschitz constant of $f$ and the constants of strong monotonicity for $f$ and $q$, c.f. \cite{ChRo97},\cite{SolSv99},\cite{LiMe79}. We see from the table above that this ratio worsen for small values of $\beta$ and large $n$ and our numerical experiments confirm these estimates. In particular, for $\beta=0.01$ and $\beta=10^{-4}$ we could not observe linear convergence neither for the FB-splitting method nor the DR-splitting method and the hybrid projection method. These methods were not able to compute an accurate solution within a reasonable time.

In the following tables we display for different combinations of $n$ and $\beta$ the mean values for the number $N$ of computed Newton directions, the number $G$ of calls to the globally convergent method and the number $F$ of evaluations of $f$ as well as the mean CPU-time in seconds needed to reach a residual less than $10^{-8}$. A time limit was set to $10^{-4}n^2$ seconds to perform this task. We tested the heuristic of Algorithm \ref{AlgSSNewtHeur}, the globally convergent hybrid algorithm \ref{AlgSSNewt_Hybrid} combined with any of the three globally convergent methods $\T^{\rm FB}_\gamma$, $\T^{\rm DR}_\gamma$ and $\T^{\rm Pr}_\gamma$ as described in Subsection \ref{SubSecSS_Hybrid}, Algorithm \ref{AlgSSNewt_DR} as well as all three globally convergent methods $\T^{\rm FB}_\gamma$, $\T^{\rm DR}_\gamma$ and $\T^{\rm Pr}_\gamma$ alone. We always choose $\nu=0.1$ and $\gamma^{(k)}=\norm{\nabla f(x^{(k)})}_1/\sqrt{n}$. In Algorithm \ref{AlgSSNewtHeur} we set $\delta^{(k)}=0.1/k$ and in Algorithm \ref{AlgSSNewt_Hybrid} we used $\delta^{(k)}\equiv 5\times 10^{-4}$. Finally, the parameter $\xi$ in Algorithm \ref{AlgSSNewt_DR} was set to $0.9$. For all test problems the origin was chosen as the starting point.

All tests were performed in MATLAB using a desktop equipped with an i7-7700 CPU, 3.6 GHZ and 32GB RAM.

\begin{table}[ht]
\[\begin{tabular}{|l|c|c|c|}
\hline
&n=150&n=600&n=2400\\
&$\begin{array}{c}\mbox{N/G/F}\\\mbox{CPU}\end{array}$&$\begin{array}{c}\mbox{N/G/F}\\\mbox{CPU}\end{array}$&$\begin{array}{c}\mbox{N/G/F}\\\mbox{CPU}\end{array}$\\
\hline
Alg.\ref{AlgSSNewt_DR}&$\begin{array}{c}5.2/5.2/41.4\\0.029\end{array}$&$\begin{array}{c}5.8/5.8/43.6\\0.394\end{array}$&$\begin{array}{c}6/6/40.6\\7.07\end{array}$\\
\hline
Alg. \ref{AlgSSNewtHeur}, Alg. \ref{AlgSSNewt_Hybrid}
&$\begin{array}{c}7.2/-/8.2\\0.012\end{array}$&$\begin{array}{c}7.6/-/8.6\\0.099\end{array}$&$\begin{array}{c}7.6/-/8.6\\2.15\end{array}$\\
\hline
$\T^{\rm FB}$&$\begin{array}{c}-/117.8/118.8\\0.164\end{array}$&$\begin{array}{c}-/158/159\\0.935\end{array}$&$\begin{array}{c}-/171.4/172.4\\7.35\end{array}$\\
\hline
$\T^{\rm DR}$&$\begin{array}{c}-/89/351.2\\0.157\end{array}$&$\begin{array}{c}-/106.2/406.2\\3.18\end{array}$&$\begin{array}{c}-/114.2/423.8\\49.8\end{array}$\\
\hline
$\T^{\rm Pr}$&$\begin{array}{c}-/487.6/978.2\\0.526\end{array}$&$\begin{array}{c}-/1322/2648\\7.84\end{array}$&$\begin{array}{c}-/1445/2894\\92.7\end{array}$\\
\hline
\end{tabular}\]
\caption{Test results for $\beta=1$}\label{TabBeta1}
\end{table}
In case when $\beta=1$ all tested methods found a solution with the prescribed tolerance within the given time limit. The heuristic Algorithm \ref{AlgSSNewtHeur} as well as the hybrid Algorithm \ref{AlgSSNewt_Hybrid} executed only Newton steps with stepsize $\alpha^{(k)}=1$, i.e., they behave like the pure \ssstar Newton method of Algorithm \ref{AlgNewton}, and showed the best performance of all methods. The second best one was Algorithm \ref{AlgSSNewt_DR} which was in turn faster than any of the three globally convergent methods.
\begin{table}[ht]
\[\begin{tabular}{|@{}l@{}|@{}c@{}|@{}c@{}|@{}c@{}|}
\hline
&n=150&n=600&n=2400\\
&$\begin{array}{c}\mbox{N/G/F}\\\mbox{CPU}\end{array}$&$\begin{array}{c}\mbox{N/G/F}\\\mbox{CPU}\end{array}$&$\begin{array}{c}\mbox{N/G/F}\\\mbox{CPU}\end{array}$\\
\hline
Alg.\ref{AlgSSNewt_DR}&$\begin{array}{c}45.8/45.8/333.4\\0.181\end{array}$&$\begin{array}{c}87.8/87.8/625\\4.34\end{array}$&$\begin{array}{c}105.8/105.8/771.4\\85.7\end{array}$\\
\hline
Alg. \ref{AlgSSNewtHeur}&$\begin{array}{c}439.6/-/1982\\1.34\end{array}$&$\begin{array}{c}986.4/-/4260\\22.4\end{array}$&$\begin{array}{c}716.2/-/2748\\265\end{array}$\\
\hline
Alg. \ref{AlgSSNewt_Hybrid}($\T^{\rm FB}$) &$\begin{array}{c}134.6/91.2/565\\0.469\end{array}$&$\begin{array}{c}482.2/42.6/1660\\9.82\end{array}$&$\begin{array}{c}624/-/1829\\209\end{array}$\\
\hline
Alg. \ref{AlgSSNewt_Hybrid}($\T^{\rm DR}$)&$\begin{array}{c}155.8/9.2/596.2\\0.393\end{array}$&$\begin{array}{c}451.2/9.6/1565\\8.95\end{array}$&$\begin{array}{c}624/-/1829\\209\end{array}$\\
\hline
Alg. \ref{AlgSSNewt_Hybrid}($\T^{\rm Pr}$)&$\begin{array}{c}165.8/9.6/593.6\\0.417\end{array}$&$\begin{array}{c}459/8.4/1564\\9.25\end{array}$&$\begin{array}{c}624/-/1829\\209\end{array}$\\
\hline
\if{$\T^{\rm FB}$&$\begin{array}{c}-/1669/1678\\2.25 (res=11.2)\end{array}$&$\begin{array}{c}-/5687/5696\\36 (res=26.3)\end{array}$&$\begin{array}{c}-/13266/13244\\576 (res=24.8)\end{array}$\\
\hline
$\T^{\rm DR}$&$\begin{array}{c}-/1621/5017\\2.25 (res=2.5\times10^{-3})\end{array}$&$\begin{array}{c}-/1421/4495\\36 (res=1.6\times10^{-2})\end{array}$&$\begin{array}{c}-/1584/5284\\576 (res=2.7\times10^{-2})\end{array}$\\
\hline
$\T^{\rm Pr}$&$\begin{array}{c}-/2144/4290\\2.25 (res=3.1\times10^{-3})\end{array}$&$\begin{array}{c}-/5685/11372\\26 (res=4.0\times10^{-4})\end{array}$&$\begin{array}{c}-/9125/18253\\576 (res=1.6\times10^{-6})\end{array}$\\
\hline}\fi
\end{tabular}\]
\caption{Test results for $\beta=10^{-2}$}\label{TabBeta01}
\end{table}

The results for $\beta=10^{-2}$ are listed in Table \ref{TabBeta01}. Algorithm \ref{AlgSSNewt_DR} was the fastest one, whereas the three globally convergent method did not find a solution within the time limit. Note that in case $n=2400$ only (damped) Newton steps were performed and therefore no calls to the global convergent method were done.
\begin{table}[ht]
\[\begin{tabular}{|@{}l@{}|@{}c@{}|@{}c@{}|@{}c@{}|}
\hline
&n=150&n=600&n=2400\\
&$\begin{array}{c}\mbox{N/G/F}\\\mbox{CPU}\end{array}$&$\begin{array}{c}\mbox{N/G/F}\\\mbox{CPU}\end{array}$&$\begin{array}{c}\mbox{N/G/F}\\\mbox{CPU}\end{array}$\\
\hline
Alg.\ref{AlgSSNewt_DR}&$\begin{array}{c}67.2/67.2/463.8\\0.271\end{array}$&$\begin{array}{c}153.4/153.4/1053\\6.45\end{array}$&$\begin{array}{c}315.8/315.8/2115\\207\end{array}$\\
\hline
Alg. \ref{AlgSSNewt_Hybrid}($\T^{\rm DR}$)&$\begin{array}{c}210.4/132/1028\\0.636\end{array}$&$\begin{array}{c}458.8/321.2/2203\\13.9\end{array}$&$\begin{array}{c}962.6/641.2/4506\\482\end{array}$\\
\hline
Alg. \ref{AlgSSNewt_Hybrid}($\T^{\rm Pr}$)&$\begin{array}{c}176.2/109/761.6\\0.514\end{array}$&$\begin{array}{c}695.4/479.4/3088\\15.6\end{array}$&$\begin{array}{c}1316/912.2/5913\\509\end{array}$\\
\hline
\end{tabular}\]
\caption{Test results for $\beta=10^{-4}$}\label{TabBeta0001}
\end{table}
For $\beta=10^{-4}$ Algorithm \ref{AlgSSNewt_DR} was again the fastest method, cf. Table \ref{TabBeta0001}. Now, in addition to the three globally convergent methods, also the heuristic Algorithm \ref{AlgSSNewtHeur} failed.

\section{Conclusion}
The theoretical background of the \ssstar Newton method has been established in \cite{GfrOut19a}. The main aim of this paper is to implement this method to a class of VIs of the second kind and to examine its numerical properties via extensive numerical experiments. The performed tests show in a convincing way that the new Newton method represents an efficient numerical tool for a number of complicated equilibrium problems. It can be used, e.g., for the computation of Nash equilibria in case of nonsmooth (and even nonconvex) objectives of the players. Further, in combination with some splitting algorithms, it exhibits remarkable (global and local) convergence properties when applied to rather complicated family of monotone variational inequalities of the second kind.
\medskip

\noindent{\bf Acknowledgements.}
The research of the first author was  supported by the Austrian Science Fund
(FWF) under grant P29190-N32. The  research of the third author was supported by the Grant Agency of the
Czech Republic and the Austrian Science Fund, project GACR-FWF  19-29646L.

\end{document}